\documentclass[12pt]{amsart}

\usepackage{amssymb,mathrsfs,bm,amsmath,amsthm}

\usepackage{enumerate}

\usepackage[centering]{geometry}
\usepackage[pagebackref]{hyperref}

\theoremstyle{plain}
\newtheorem{thm}{Theorem}[section]
\newtheorem{prop}[thm]{Proposition}
\newtheorem{lem}[thm]{Lemma}
\theoremstyle{definition}
\newtheorem{defn}{Definition}[section]
\newtheorem{exmp}{Example}[section]
\theoremstyle{remark}

\numberwithin{equation}{section}

\DeclareMathOperator{\hdim}{\dim_H}
\DeclareMathOperator{\pdim}{\dim_P}
\DeclareMathOperator{\dist}{dist}
\renewcommand{\Im}{\operatorname{Im}}
\renewcommand{\Re}{\operatorname{Re}}

\newcommand{\Q}{\mathbb Q}
\newcommand{\N}{\mathbb N}
\newcommand{\R}{\mathbb R}
\newcommand{\Z}{\mathbb Z}
\def\C{\mathbb C}

\newcommand{\D}{\mathop{}\!\mathrm{d}}
\newcommand{\dd}{\mathfrak d}
\newcommand{\lm}{\mathcal L}
\newcommand{\hdm}{\mathcal H}
\newcommand{\cst}{\mathfrak c}
\newcommand{\cld}{\mathcal C}
\newcommand{\dmn}{\mathfrak F}
\newcommand{\dmng}{\mathfrak G}

\begin{document}
	
\title[The difference between the Hurwitz continued fraction expansions]{The difference between the Hurwitz continued fraction expansions of a complex number and its rational approximations}

\author{YuBin He}

\address{Department of Mathematics, South China University of Technology,	Guangzhou, 510641, P.~R.\ China}

\email{yubinhe007@gmail.com}

\author[Ying Xiong]{Ying Xiong$^*$}

\address{Department of Mathematics, South China University of Technology,	Guangzhou, 510641, P.~R.\ China}

\email{xiongyng@gmail.com}



\thanks{$^*$Corresponding author.}

\thanks{This work is supported by National Natural Science Foundation of China (Grant No.~11871227, 11771153, 11471124), Guangdong Natural Science Foundation (Grant No.~2018B0303110005), Guangdong Basic and Applied Basic Research Foundation (Project No. 2021A1515010056), the Fundamental Research Funds for the Central Universities, SCUT (Grant No.~2020ZYGXZR041).}

\begin{abstract}
	For regular continued fraction, if a real number~$x$ and its rational approximation~$p/q$ satisfying $|x-p/q|<1/q^2$, then, after deleting the last integer of the partial quotients of~$p/q$, the sequence of the remaining partial quotients is a prefix of that of~$x$. In this paper, we show that the situation is completely different if we consider the Hurwitz continued fraction expansions of a complex number and its rational approximations.
	More specifically, we consider the set $E(\psi)$ of complex numbers which are well approximated with the given bound $\psi$ and have quite different Hurwitz continued fraction expansions from that of their rational approximations. The Hausdorff and packing dimensions of such set are determined. It turns out that its packing dimension is always full for any given approximation bound $\psi$ and its Hausdorff dimension is equal to that of the $\psi$-approximable set $W(\psi)$ of complex numbers. As a consequence, we also obtain an analogue of the classical Jarn\'ik Theorem in real case.
\end{abstract}

\maketitle

\section{Introduction}

The theory of continued fractions is important for studying the problem of approximating real numbers by rational fractions. Given $x\in(0,1)$, we can write it as the expression
\begin{equation}\label{eq:CF}
	x=\cfrac{1}{a_1+\cfrac{1}{a_2+\cfrac{1}{a_3+\cdots}}}=:[0;a_1,a_2,a_3,\dotsc],
\end{equation}
where $(a_n)$, called the partial quotients of \emph{regular continued fraction} (RCF) of~$x$, is given by $a_n=\lfloor1/G^{n-1}(x)\rfloor$. Here $G(x)=1/x-\lfloor1/x\rfloor$ is the Gauss map and $\lfloor x\rfloor$ is the greatest integer function which stands for the greatest integer less than or equal to~$x$. If $x\in\Q$, then we obtain $G^N(x)=0$ for some~$N\ge1$ and the partial quotients of~$x$ are $(a_1,\dots,a_N)$. Otherwise, $(a_n)$ is an infinite sequence. We remark that by this definition, the continued fraction of a rational number could not be of form $ [0;a_1,\dots,a_{N-1},1] $, i.e., ending with 1, since we will have $ G^{N-1}(x)=0 $.

In 1887, A.~Hurwitz~\cite{Hurwi87} considered the analogous problem of approximating complex numbers by ratios of Gauss integers. For this, the \emph{Hurwitz continued fraction} (HCF) was introduced in a similar way, with the nearest integer function $[z]$ instead of the greatest integer function $\lfloor x\rfloor$. 

The nearest Gaussian integer of a complex number~$z$, denoted by~$[z]$, is determined by
\[ z-[z]\in\dmn:=\{x+iy\colon x,y\in[-1/2,1/2)\}. \]
Write $\dmn^*=\dmn\setminus\{0\}$. For $z\in\dmn^*$, define $a_n=[1/T^{n-1}(z)]$, where $T(z)=1/z-[1/z]$. The sequence $(a_n)$ is called the partial quotients of Hurwitz continued fraction, since, for $z\in\dmn^*$, it leads to the same expansion as~\eqref{eq:CF}:
\[ z=[0;a_1,a_2,a_3,\dotsc]. \]
If $z\in\Q(i)$, then we obtain $T^N(z)=0$ for some~$N\ge1$ and the partial quotients of~$z$ are $(a_1,\dots,a_N)$. Otherwise, $(a_n)$ is an infinite sequence. Moreover, since $[1/z]\notin\{0,\pm1,\pm i\}$ for $z\in\dmn^*$, we have
\begin{equation}\label{eq:an}
	a_n\in\Z[i]\setminus\{0,\pm1,\pm i\}=\bigl\{a\in\Z[i]\colon|a|\ge\sqrt2\bigr\}=:I.
\end{equation}
We write $a_n(z)$ instead of $a_n$ when it is necessary to emphasis that $(a_n)$ is obtained from $z$. 

These two continued fractions have many similar properties related to the Diophantine approximation (see for instance~\cite{Hensl06}). It is also interesting to find the difference between them. In this paper, we discuss a property that holds for the RCF but does not hold for the HCF.

Let us consider an irrational number $x\in(0,1)$ and its rational approximation $p/q$ such that $|x-p/q|<1/q^2$. Let $(a_n)_{n\ge1}$ and $(b_1,\dots,b_N)$ be the RCF partial quotients of~$x$ and $p/q$, respectively. Then we must have $a_k=b_k$ for $1\le k\le N-1$ (see Proposition~\ref{p:RCF}). In other word, if we delete the last integer of the RCF partial quotients of~$p/q$, then the remaining partial quotients is a prefix of that of~$x$.

However, the HCF partial quotients of a complex number and its rational approximations may be quite different. For example, let $z=\frac{2i}{3-\sqrt{10}+7i}$ and $p/q=\frac{37+6i}{129+4i}$. Then $|z-p/q|\approx0.000029<0.000058\approx1/|q|^2$, but the HCF partial quotients of~$z$ and~$p/q$ are  $(4,\overline{-2,3i,2,3i})$ and $(3,2,3i,-2,3i)$, respectively, which are quite different. Here $\overline {a_1,\dots,a_k}$ denotes the periodic sequence $(a_1,\dots,a_k,\dots,a_1,\dots,a_k,\dots)$.

This is not accidental. In fact, we shall prove a general result which asserts that even if $z$ and $p/q$ are much more close than $1/|q|^2$, one can still find a lot of such examples. 

To be more precise, for $z\notin\Q(i)$ with the HCF partial quotients $(a_n)_{n\ge1}$ and $p/q\in\Q(i)$ with the HCF partial quotients $(b_1,\dots,b_N)$, we use the cardinal number
\begin{equation}\label{eq:dd}
	\dd(z,p/q)=\sharp\{n\colon a_n\ne b_n, 1\le n\le N\}
\end{equation}
to quantify the difference between the partial quotients of~$z$ and~$p/q$. More generally, we use $\psi(|q|)$ instead of $1/|q|^2$ to bound the distance of $z$ and $p/q$, where $\psi\colon(0,\infty)\to(0,\infty)$ is a nonincreasing function with $\lim_{x\to+\infty}\psi(x)=0$. Let $E(\psi)$ be the set of points $z\in\dmn^*$ such that there exist infinitely many $p^{(n)}/q^{(n)}\in\Q(i)$ satisfying $|z-p^{(n)}/q^{(n)}|\le\psi(|q^{(n)}|)$ and $\dd(z,p^{(n)}/q^{(n)})\to+\infty$ as $n\to\infty$. Our first result gives the Hausdorff and packing dimensions of~$E(\psi)$.
\begin{thm}\label{t:result}
	We have
	\[ \hdim E(\psi)=\min\bigl(4/\lambda(\psi),2\bigr) \quad\text{and}\quad
	\pdim E(\psi)=2. \]
	Here $\lambda(\psi)=\liminf_{x\to\infty}(-\log\psi(x))/\log x$. 
\end{thm}

The lower order $\lambda(\psi)$ appears naturally in Dodson's works~\cite{Dodso91,Dodso92} on the Hausdorff dimension of $\psi$-approximable sets, which are well studied for the cases of real numbers and linear forms, see for instance~\cite{Besic34,Dicki97,Dodso91,Dodso92,Jarn29,Jarn31,KimLi19,WanWu17}. 

To make clear the meaning of the dimensions result in Theorem~\ref{t:result}, we consider the $\psi$-approximable set of complex numbers
\begin{equation}\label{eq:wpsi}
	W(\psi):=\bigl\{z\in\dmn^*\colon|z-p/q|\le\psi(|q|)\ \text{for infinitely many $p,q\in\Z[i]$}\bigr\}.
\end{equation}
Our second result says that the Hausdorff and packing dimensions of~$W(\psi)$ are equal to that of~$E(\psi)$. This is an analogue of the classical Jarn\'ik Theorem in real case.
\begin{thm}\label{t:Wpsi}
	$\hdim W(\psi)=\min\bigl(4/\lambda(\psi),2\bigr)$ and $\pdim W(\psi)=2$.
\end{thm}

Obviously, we have $E(\psi)\subset W(\psi)$. Therefore, Theorem~\ref{t:result} and~\ref{t:Wpsi} imply that the Hausdorff and packing dimensions of~$E(\psi)$ both attain their largest possible value. This means that, in terms of the Hausdorff or packing dimension, it is a common phenomenon that a complex number and its rational approximations have quite different HCF partial quotients, though the situation is completely different if we consider the RCF partial quotients for a real number and its rational approximations.

This paper is organized as follows. Section~\ref{sec:pre} introduces some preliminaries. More specifically, in Section~\ref{ss:RCF}, we discuss the relation between the RCF partial quotients of an irrational number and its rational approximations; in Section~\ref{ss:HCF}, we introduce some standard facts of HCF; in Section~\ref{ss:BRC}, we present two lemmas which are used to find complex irrational numbers having quite different HCF partial quotients compared with their rational approximations; in Section~\ref{ss:AF}, we give some useful lemmas which are needed to obtain the lower bounds of dimensions of~$E(\psi)$. Section~\ref{sec:proof} is devoted to the proofs of Theorems~\ref{t:result} and~\ref{t:Wpsi}.

\section{Preliminaries} \label{sec:pre}

\subsection*{Notation}
Throughout this paper, we adhere to the following notation. Sequences of numbers will be denoted by letters in boldface: $\bm a,\bm b,\dotsc$; for $\bm a=(a_1,\dots,a_n)$ and $\bm b=(b_1,\dots,b_m)$, write $\bm a^-=(a_1,\dots,a_{n-1})$ and $\bm{ab}=(a_1,\dots,a_n,b_1,\dots,b_m)$; the empty word and empty set will be denoted by~$\varnothing$ and~$\emptyset$, respectively; the open and closed ball with center~$z$ and radius~$r$ will be denoted by~$B(z,r)$ and $\overline B(z,r)$, respectively; the conjugation, real and imaginary part of a complex number~$z$ will be denoted by~$\overline z$, $\Re z$ and~$\Im z$, respectively; the interior, closure, diameter, cardinal number and Lebesgue measure of~$A$ will be denoted by~$A^\circ$, $\overline A$, $|A|$, $\sharp A$ and~$\lm(A)$, respectively. For two variables~$\alpha$ and~$\beta$, the notation $\alpha\asymp\beta$ means that $c^{-1}\alpha\le\beta\le c\alpha$ for some constant~$c\ge1$.

Define $p(\varnothing)=0$ and $q(\varnothing)=1$ for the empty word $\varnothing$. Given a sequence $(a_1,\dots,a_n)$ of numbers with $n\ge1$, define
\begin{equation}\label{eq:Qpair}
	\begin{pmatrix}
		p(a_1,\dots,a_n) & p(a_1,\dots,a_{n-1}) \\ 
		q(a_1,\dots,a_n) & q(a_1,\dots,a_{n-1})
	\end{pmatrix}
	=\begin{pmatrix}
		0 & 1 \\ 
		1 & 0
	\end{pmatrix}\cdot\prod_{j=1}^{n}
	\begin{pmatrix}
		a_j & 1 \\ 
		1 & 0
	\end{pmatrix}.
\end{equation}
Here we regard $(a_1,\dots,a_{n-1})$ as the empty word~$\varnothing$ when $n=1$. For~$z\in\dmn^*$ (or $x\in(0,1)$) with the HCF (or RCF) partial quotients $(a_n)$, let $p_n=p(a_1,\dots,a_n)$ and $q_n=q(a_1,\dots,a_n)$ with the convention $p_0=p(\varnothing)=0$ and $q_0=q(\varnothing)=1$. By~\eqref{eq:Qpair}, we have the recursive formulae:
\[p_n=a_np_{n-1}+p_{n-2}\quad\textrm{and}\quad q_n=a_nq_{n-1}+q_{n-2},\qquad \text{for}\ n\ge 2.\]
Moreover, it holds that ${p_n}/{q_n}=[0;a_1,a_2,\dots,a_n]$. We call the sequences $(p_n)$ and $(q_n)$ the \emph{$\mathcal Q$-pair} of~$z$ (or~$x$).

\subsection{The RCF partial quotients} \label{ss:RCF}

This subsection is devoted to the relationship between the RCF partial quotients of an irrational number and its rational approximations.
\begin{prop}\label{p:RCF}
	For $x\in(0,1)\setminus\Q$ and $p,q\in\N$ with $|x-p/q|<1/q^2$, let $(a_n)_{n\ge1}$ and $(b_1,\dots,b_N)$ be the RCF partial quotients of~$x$ and~$p/q$, respectively. Then we have $a_k=b_k$ for $1\le k\le N-1$.
\end{prop}
\begin{proof}
We have not found the exact statement of the proposition in the literature, although there is a closely related result due to Fatou~\cite{Fatou04} (see also Grace~\cite{Grace18}). For reader's convenience, we give a direct proof here.

Write $p_k=p(b_1,\dots,b_k)$ and $q_k=q(b_1,\dots,b_k)$ for $k=N,N-1,N-2$ (recall~\eqref{eq:Qpair}). Since the function $x\mapsto(ax+b)/(cx+d)$ is monotonic on $ [0,\infty) $, one sees that
\[ \frac{p}{q}=\frac{p_N}{q_N}=\frac{b_Np_{N-1}+p_{N-2}}{b_Nq_{N-1}+q_{N-2}} \]
lies between $p_{N-1}/q_{N-1}$ and $(p_{N-1}+p_{N-2})/(q_{N-1}+q_{N-2})$. We claim that so does~$x$. Consequently, for some $m\ge1$, $x$ lies between
\[ \frac{mp_{N-1}+p_{N-2}}{mq_{N-1}+q_{N-2}}=[0;b_1,\dots,b_{N-1},m] \]
and
\[\frac{(m+1)p_{N-1}+p_{N-2}}{(m+1)q_{N-1}+q_{N-2}}=[0;b_1,\dots,b_{N-1},m+1],\]
this means that the first $N$ partial quotients of~$x$ are $(b_1,\dots,b_{N-1},m)$.

It remains to prove the claim. By the definition of the continued fraction of a rational number, its partial quotients could not be end with $ 1 $, we must have $ b_N\ge 2 $. Hence,
\[ q=q_N=b_Nq_{N-1}+q_{N-2}>q_{N-1}+q_{N-2}>q_{N-1}. \]
Therefore,
\[ \min\left(\left|\frac{p}{q}-\frac{p_{N-1}}{q_{N-1}}\right|,\left|\frac{p}{q}-\frac{p_{N-1}+p_{N-2}}{q_{N-1}+q_{N-2}}\right|\right)>1/q^2. \]
This together with the condition $|x-p/q|<1/q^2$ gives the claim.
\end{proof}

\subsection{Basic properties of HCF and regular cylinders} \label{ss:HCF}

This subsection presents some standard facts of HCF. We think all of these results are known, however fail to find references for some of them, so we include proofs for the reader's convenience.

\begin{lem}\label{l:props}
	Let $z\in\dmn^*$ and denote its HCF partial quotients by $(a_n)$. Let $(p_n)$ and $(q_n)$ be the $\mathcal Q$-pair of~$z$. For all $n\ge1$, the following statements hold.
	\begin{enumerate}[\upshape(a)]
		\item \label{le:mf} The mirror formula: $q_{n-1}/q_n=[0;a_n,a_{n-1},\dots,a_1]$. 
		\item \label{le:pq-pq} $q_np_{n-1}-q_{n-1}p_n=(-1)^n$. 
		\item \label{le:z-pq} $|z-p_n/q_n|=\bigl|q_n^2(a_{n+1}+T^{n+1}(z)+q_{n-1}/q_n)\bigr|^{-1}\le|q_n|^{-2}$. 
		\item \label{le:1<q} $1=|q_0|<|q_1|<|q_2|<\dotsb$.
		\item \label{le:q>q} $|q_{n+k}|\ge\phi^{\lfloor n/2\rfloor}|q_k|$ for $k\ge0$. In particular, $|q_n|\ge\phi^{\lfloor n/2\rfloor}$. Here $\phi=\frac{\sqrt5+1}{2}$ and $\lfloor x\rfloor$ denotes the greatest integer $\le x$. 
		\item \label{le:|q-|} Let $\bm a=(a_1,\dots,a_n)$, then $(|a_n|-1)|q(\bm a^-)|<|q(\bm a)|<(|a_n|+1)|q(\bm a^-)|$.
		\item \label{le:|qq|} Let $\bm a=(a_1,\dots,a_n)$, $\bm b=(a_{n+1},\dots,a_{n+k})$ with $k\ge1$. Then 
		\[ |q(\bm a)q(\bm b)|/5<|q(\bm{ab})|< 3|q(\bm a)q(\bm b)|. \]
	\end{enumerate}
\end{lem}
\begin{proof}
	For~\eqref{le:mf}, see~\cite[page~76]{Hensl06}; for~\eqref{le:pq-pq}, see~\cite[page~73]{Hensl06}; for~\eqref{le:z-pq}, see~\cite[Theorem~1]{Lakei73}; for~\eqref{le:1<q}, see~\cite{Hurwi87}; for~\eqref{le:q>q}, see~\cite[Corollary~5.3]{DanNo14}.
	
	(f) If $n=1$, then $q(\bm a)=a_1$ and $q(\bm a^-)=1$, so the conclusion holds. Suppose $n\ge2$. By~\eqref{eq:Qpair}, $q(\bm a)=a_nq(\bm a^-)+q(a_1,\dots,a_{n-2})$. Since $|q(a_1,\dots,a_{n-2})|<|q(\bm a^-)|$, the conclusion follows.
	
	(g) It follows from~\eqref{eq:Qpair} that
	\[ q(\bm{ab})=q(\bm a)q(\bm b)+q(\bm a^-)p(\bm b). \]
	By~\eqref{le:1<q}, $|q(\bm a^-)|<|q(\bm a)|$. So
	\begin{equation}\label{eq:q/qq}
		1-\left|\frac{p(\bm b)}{q(\bm b)}\right|<\frac{|q(\bm{ab})|}{|q(\bm a)q(\bm b)|}<1+\left|\frac{p(\bm b)}{q(\bm b)}\right|.
	\end{equation}
	We now turn to bound $|p(\bm b)/q(\bm b)|$. Note that $(a_{n+k})_{k}$ is the HCF of~$T^n(z)\in\dmn$. Conclusion~\eqref{le:z-pq} gives $|T^n(z)-p(\bm b)/q(\bm b)|\le|q(\bm b)|^{-2}$, and so
	\[ \left|\frac{p(\bm b)}{q(\bm b)}\right|\le|T^n(z)|+\frac{1}{|q(\bm b)|^2}\le\frac{\sqrt2}{2}+\frac{1}{|q(\bm b)|^2}<2, \]
	since $|q(\bm b)|>1$ (by~\eqref{le:1<q}). From~\eqref{eq:q/qq}, we obtain $|q(\bm{ab})|< 3|q(\bm a)q(\bm b)|$.
	
	To prove another inequality, assume that $|q(\bm b)|>4$, then $|p(\bm b)/q(\bm b)|\le\sqrt2/2+|q(\bm b)|^{-2}<4/5$. By~\eqref{eq:q/qq}, we have $|q(\bm a)q(\bm b)|/5<|q(\bm{ab})|$. Otherwise, if $|q(\bm b)|\le4$, then $|q(\bm a)q(\bm b)|/5<|q(\bm a)|<|q(\bm{ab})|$ by~\eqref{le:1<q}.
\end{proof}

Recall from~\eqref{eq:an} that $I=\{a\in\Z[i]\colon|a|\ge\sqrt2\}$. A \emph{cylinder} of level~$n$ is a subset of~$\dmn$ taking the form
\[ \cld(\bm u)=\bigl\{z\in\dmn\colon a_1(z)=u_1,\dots,a_n(z)=u_n\bigr\}, \]
where $\bm u=(u_1,\dots,u_n)\in I^n$. We adopt the convention that $I^0=\{\varnothing\}$ and $\cld(\varnothing)=\dmn$. One can check that $\{\cld(\bm u)\}_{\bm u\in I^n}$ forms a partition of~$\dmn$ except for some points in $\Q(i)$ for each~$n\ge0$. 

To study the cylinders~$\cld(\bm u)$ of level~$n$, let $T_{\bm u}=(T^n|_{\cld(\bm u)})^{-1}$ with the convention $T_\varnothing=\mathrm{identity}$. We call $\dmn_{\bm u}=T^n\bigl(\cld(\bm u)\bigr)=T_{\bm u}^{-1}\bigl(\cld(\bm u)\bigr)$ the \emph{prototype set} of~$\cld(\bm u)$. Note that $T_{\bm u}\colon\dmn_{\bm u}\to\cld(\bm u)$ is the M\"obius transformation of the form 
\begin{equation}\label{eq:Tu}
	T_{\bm u}\colon z\mapsto\frac{p(\bm u^-)z+p(\bm u)}{q(\bm u^-)z+q(\bm u)}\quad\text{or equivalently}\ T_{\bm u}\colon [0;\bm v]\mapsto[0;\bm{uv}].
\end{equation}

The lemma below is very useful in computing cylinders and their prototype sets.
\begin{lem}\label{l:image}
	Let $\bm u\in I^n$ and $\bm v\in I^m$ with $\cld(\bm u)\ne\emptyset$ and $\cld(\bm v)\ne\emptyset$. Let $T_{\bm u}$ and~$T_{\bm v}$ be as in~\eqref{eq:Tu}.
	\begin{enumerate}[\upshape(a)]
		\item \label{le:Cuv} $\cld(\bm{uv})=T_{\bm u}\bigl(\dmn_{\bm u}\cap\cld(\bm v)\bigr)$ and $\dmn_{\bm{uv}}=T_{\bm v}^{-1}(\dmn_{\bm u})\cap\dmn_{\bm v}$.
		\item \label{le:Fub} If $\bm v=b\in I$, then $\dmn_{\bm ub}=T_b^{-1}(\dmn_{\bm u})\cap\dmn$.
		\item \label{le:uv=u'v} If $\dmn_{\bm u}=\dmn_{\bm u'}$ for some sequence $\bm u'\in I^{n'}$, then $\dmn_{\bm{uv}}=\dmn_{\bm u'\bm v}$.
		\item \label{le:image} If $\cld(\bm v)\subset\dmn_{\bm u}$, then $\cld(\bm{uv})=T_{\bm u}\bigl(\cld(\bm v)\bigr)$ and $\dmn_{\bm{uv}}=\dmn_{\bm v}$. In particular, this holds if $\dmn_{\bm u}=\dmn$.
	\end{enumerate}
\end{lem}
\begin{proof}
	(a) Note that $z\in\cld(\bm{uv})$ if and only if $z\in\cld(\bm u)$ and $T^n(z)\in\cld(\bm v)$. This is also equivalent to $z\in\cld(\bm u)\cap T_{\bm u}\bigl(\cld(\bm v)\bigr)$. Hence,
	\[ \cld(\bm{uv})=\cld(\bm u)\cap T_{\bm u}\bigl(\cld(\bm v)\bigr)=T_{\bm u}\circ T^n\bigl(\cld(\bm u)\bigr)\cap T_{\bm u}\bigl(\cld(\bm v)\bigr)=T_{\bm u}\bigl(\dmn_{\bm u}\cap\cld(\bm v)\bigr). \]
	Consequently, 
	\begin{align*}
		\dmn_{\bm{uv}} &=T^{n+m}\bigl(\cld(\bm{uv})\bigr) =T^m\circ T^n\circ T_{\bm u}\bigl(\dmn_{\bm u}\cap\cld(\bm v)\bigr)\\
		&=T_{\bm v}^{-1}\bigl(\dmn_{\bm u}\cap\cld(\bm v)\bigr) =T_{\bm v}^{-1}(\dmn_{\bm u})\cap\dmn_{\bm v}.
	\end{align*}
	
	(b) By~\eqref{le:Cuv}, it suffices to show that $T_b^{-1}(\dmn_{\bm u})\cap\dmn=T_b^{-1}(\dmn_{\bm u})\cap\dmn_b$. The ``$\supset$'' part is obvious since $\dmn\supset\dmn_b$. Now pick $z\in T_b^{-1}(\dmn_{\bm u})\cap\dmn$, then $T_b(z)=(z+b)^{-1}\in\dmn_{\bm u}\subset\dmn$. This together with $z\in\dmn$ gives $T_b(z)\in\cld(b)$, and so $z\in T\bigl(\cld(b)\bigr)=\dmn_b$. This implies $z\in T_b^{-1}(\dmn_{\bm u})\cap\dmn_b$ and completes the proof of ``$\subset$'' part.
	
	(c) By~\eqref{le:Cuv}, $\dmn_{\bm{uv}}=T_{\bm v}^{-1}(\dmn_{\bm u})\cap\dmn_{\bm v}=T_{\bm v}^{-1}(\dmn_{\bm u'})\cap\dmn_{\bm v}=\dmn_{\bm u'\bm v}$.
	
	(d) By~\eqref{le:Cuv}, if $\cld(\bm v)\subset\dmn_{\bm u}$, then $\cld(\bm{uv})=T_{\bm u}\bigl(\dmn_{\bm u}\cap\cld(\bm v)\bigr)=T_{\bm u}\bigl(\cld(\bm v)\bigr)$, and so
	\[ \dmn_{\bm{uv}}=T^{n+m}\bigl(\cld(\bm{uv})\bigr) =T^m\circ T^n\circ T_{\bm u}\bigl(\cld(\bm v)\bigr)=T^m\bigl(\cld(\bm v)\bigr)=\dmn_{\bm v}. \]
	In particular, we always have $\cld(\bm v)\subset\dmn_{\bm u}$ if $\dmn_{\bm u}=\dmn$.
\end{proof}

Using Lemma~\ref{l:image}~\eqref{le:Fub} repeatedly, we obtain the following two examples of certain prototype sets, which are needed later.
\begin{exmp}\label{e:-2i}
	$\dmn_{-2i,-2,2i,-2,-2i}=\dmn_{-2i}=\dmn\setminus\overline B(i,1)$. 
\end{exmp}
The second equality follows from~\cite[page 8]{Gonza18}, though the notations we used are diffierent, the meaning is essentially the same. It is worth noting that
\begin{align*}
	\{z\in\C:\Im z=1/2\}^{-1}&=\{z\in\C:|z+i|=1\},\\
	\{z\in\C:\Re z=1/2\}^{-1}&=\{z\in\C:|z-1|=1\},\\
	\{z\in\C:\Im z=-1/2\}^{-1}&=\{z\in\C:|z-i|=1\},\\
	\{z\in\C:\Re z=-1/2\}^{-1}&=\{z\in\C:|z+1|=1\}.
\end{align*} 
Here $A^{-1}=\{z^{-1}\colon z\in A\}$. Computing step by step, we have
\begin{align*}
	\dmn_{-2i,-2}&=T_{-2}^{-1}(\dmn_{-2i})\cap\dmn=(\dmn_{-2i}^{-1}+2)\cap\dmn\\
	&=\bigg(\Big(\{z\in\C:\Im z>-1/2\}\setminus \big(B(-1,1)\cup B(i,1)\cup\bar B(1,1)\big)\Big)+2\bigg)\cap \dmn\\
	&=\Big(\{z\in\C:\Im z>-1/2\}\setminus \big(B(1,1)\cup B(2+i,1)\cup\bar B(3,1)\big)\Big)\cap \dmn\\
	&=(\dmn\cap\{z\in\C\colon\Im z>-1/2\})\setminus B(1,1).
\end{align*}
\begin{align*}
	\dmn_{-2i,-2,2i}&=T_{2i}^{-1}(\dmn_{-2i,-2})\cap\dmn=(\dmn_{-2i,-2}^{-1}-2i)\cap \dmn\\
	&=\bigg(\Big(\{z\in\C:\Re z\le1/2\}\setminus \big(B(-1,1)\cup \bar{B}(i,1)\cup\bar B(-i,1)\big)\Big)-2i\bigg)\cap \dmn\\
	&=\Big(\{z\in\C:\Re z\le1/2\}\setminus \big(B(-1-2i,1)\cup \bar{B}(-i,1)\cup\bar B(-3i,1)\big)\Big)\cap \dmn\\
	&=(\dmn\cap\{z\in\C\colon\Re z\le1/2\})\setminus \bar B(-i,1)\\
	&=\dmn\setminus\bar B(-i,1).
\end{align*}
\begin{align*}
	\dmn_{-2i,-2,2i,-2}&=T_{-2}^{-1}(\dmn_{-2i,-2,2i})\cap\dmn=(\dmn_{-2i,-2,2i}^{-1}+2)\cap \dmn\\
	&=\bigg(\Big(\{z\in\C:\Im z<1/2\}\setminus \big(B(-1,1)\cup \bar{B}(1,1)\cup\bar B(-i,1)\big)\Big)+2\bigg)\cap \dmn\\
	&=\Big(\{z\in\C:\Im z<1/2\}\setminus \big(B(1,1)\cup \bar{B}(3,1)\cup\bar B(2-i,1)\big)\Big)\cap \dmn\\
	&=\dmn\setminus B(1,1).
\end{align*}
\begin{align*}
	\dmn_{-2i,-2,2i,-2,-2i}&=T_{-2i}^{-1}(\dmn_{-2i,-2,2i,-2})\cap\dmn=(\dmn_{-2i,-2,2i,-2}^{-1}+2i)\cap \dmn\\
	&=\bigg(\Big(\{z\in\C:\Re z\le1/2\}\setminus \big(B(-1,1)\cup B(i,1)\cup\bar B(-i,1)\big)\Big)+2i\bigg)\cap \dmn\\
	&=\Big(\{z\in\C:\Re z\le1/2\}\setminus \big(B(-1+2i,1)\cup B(3i,1)\cup\bar B(i,1)\big)\Big)\cap \dmn\\
	&=\dmn\setminus \bar B(i,1).
\end{align*}

\begin{exmp}\label{e:2i-}
	$\dmn_{2i,-2+i,2i,-2+i,2i}=\dmn_{2i,-2+i,2i}=\dmn\cap\{z\in\C\colon|z+i|=1\}$.
\end{exmp}
The calculations are similar to Example~\ref{e:-2i}, so we give the results directly and omit the details.
\begin{align*}
&\dmn_{2i}=T_{2i}^{-1}(\dmn)\cap\dmn=\dmn\setminus B(-i,1),\\
&\dmn_{2i,-2+i}=T_{-2+i}^{-1}(\dmn_{2i})\cap\dmn=\{t-\tfrac{i}{2}\colon t\in[-\tfrac{1}{2},1-\tfrac{\sqrt{3}}{2}]\},\\
&\dmn_{2i,-2+i,2i}=T_{2i}^{-1}(\dmn_{2i,-2+i})\cap\dmn=\dmn\cap\{z\in\C\colon|z+i|=1\},\\
&\dmn_{2i,-2+i,2i,-2+i}=T_{-2+i}^{-1}(\dmn_{-2i,-2,2i})\cap\dmn\\
&\mspace{106mu}=\{t+\tfrac{i}{2}\colon t\in[-\infty,-1-\tfrac{\sqrt{3}}{2}]\cup[1+\tfrac{\sqrt{3}}{2},\infty]\},\\
&\dmn_{2i,-2+i,2i,-2+i,2i}=T_{2i}^{-1}(\dmn_{-2i,-2,2i,-2})\cap\dmn=\dmn\cap\{z\in\C\colon|z+i|=1\}.
\end{align*}
\begin{defn}[see~{\cite[\S1.2, \S3.1]{Gonza18}}]
	We say $\cld(\bm u)$ or~$\bm u$ is regular if $\dmn_{\bm u}$ has nonempty interior, otherwise it is said to be irregular. Moreover, we say $\cld(\bm u)$ or $\bm u$ is full if $\dmn_{\bm u}=\dmn$.
\end{defn}

\begin{exmp}[see~{\cite[page~8]{Gonza18}}]\label{e:reg}
	All the cylinders of level~$1$ are regular. Moreover, a cylinder~$\cld(b)$ of level~$1$ is full if and only if $|b|\ge2\sqrt2$.
\end{exmp}

The importance of regular cylinders is based on the fact that they take up all the Lebesgue measure.

\begin{lem}[see {\cite[Proposition~1.2.1]{Gonza18}}]\label{l:RCmeas}
	The Lebesgue measure of the union of all irregular cylinders is zero.
\end{lem}

The following lemma asserts that there are only finitely many different prototype sets for regular sequences.
\begin{lem}[see~{\cite[Lemma~1]{EiINN19}} or~{\cite[\S3]{HiaVa18}}]\label{l:Ftype}
	For all regular sequences~$\bm u$, there are only finitely many different possibilities for the prototype sets~$\dmn_{\bm u}$. More precisely, $\dmn_{\bm u}^\circ$ must be one of the following 13 different sets: $\dmn^\circ$ or $\{i^j\dmng_k\}_{0\le j\le3,1\le k\le3}$, where $\dmng_1=(\dmn\setminus(B(1,1)\cup B(i,1)))^\circ$, $\dmng_2=(\dmn\setminus B(1,1))^\circ$ and $\dmng_3=(\dmn\setminus B(1+i,1))^\circ$.
\end{lem}

The lemma below tells us how to obtain new regular or full cylinders from known ones.
\begin{lem}\label{l:newRC}
	Let $\bm u=(u_1,\dots,u_n)\in I^n$ and $\bm v=(v_1,\dots,v_m)\in I^m$.
	\begin{enumerate}[\upshape(a)]
		\item \label{le:ur} If $\cld(\bm u)$ is regular, then so is $\cld(u_1,\dots,u_j)$ for $1\le j\le n$.
		\item \label{le:ubF} If $\cld(\bm u)$ is regular, then $\cld(\bm ub)$ is full for some $b\in\{\pm2\pm2i\}$.
		\item \label{le:usr} If $\cld(\bm u)$ is full, then so is $\cld(u_{j+1},\dots,u_n)$ for $0\le j<n$.
		\item \label{le:uvsr} If both $\cld(\bm u)$ and $\cld(\bm v)$ are full, then so is $\cld(\bm{uv})$.
	\end{enumerate}
\end{lem}
\begin{proof}
	(a) Note that the restriction of~$T^n$ to the cylinder $\cld(\bm u)$ is a M\"obius transformation. Thus, a cylinder is regular if and only if it has nonempty interior. So the conclusion follows from $\cld(u_1,\dots,u_j)\supset\cld(\bm u)$.
	
	(b) Let $\{\dmng_k\}_{1\le k\le3}$ be as in Lemma~\ref{l:Ftype}. Note that $\dmng_1\subset\dmng_2\subset\dmng_3\subset\dmn$. Thus, given a regular sequence~$\bm u$, Lemma~\ref{l:Ftype} says that $i^j\dmng_1\subset\dmn_{\bm u}$ for some $j\in\{0,1,2,3\}$. One can check that 
	\[ \cld(-2+2i)\subset\dmng_1\quad\text{and}\quad i^j\cld(-2+2i)=\cld\bigl((-i)^j\cdot(-2+2i)\bigr). \]
	Hence, we have $\cld(b)\subset i^j\dmng_1\subset\dmn_{\bm u}$ for $b=(-i)^j\cdot(-2+2i)\in\{\pm2\pm2i\}$. Finally, Lemma~\ref{l:image}~\eqref{le:image} gives $\dmn_{\bm ub}=\dmn_b=\dmn$, since $\cld(b)$ is full by Example~\ref{e:reg}.

	(c) If $z\in\cld(\bm u)$, then $T^j(z)=[0;u_{j+1},\dots,u_n,\dotsc]\in\cld(u_{j+1},\dots,u_n)$. Hence,
	\[ \dmn=T^n\bigl(\cld(\bm u)\bigr)=T^{n-j}\circ T^j\bigl(\cld(\bm u)\bigr)\subset T^{n-j}\bigl(\cld(u_{j+1},\dots,u_n)\bigr)\subset\dmn. \]
	This clearly forces $\dmn_{u_{j+1},\dots,u_n}=T^{n-j}\bigl(\cld(u_{j+1},\dots,u_n)\bigr)=\dmn$.
	
	(d) By Lemma~\ref{l:image}~\eqref{le:image}, we have $\dmn_{\bm{uv}}=\dmn_{\bm v}=\dmn$.
\end{proof}

We want to estimate the diameter and Lebesgue measure of regular cylinders. Lemma~\ref{l:props}~\eqref{le:z-pq} gives upper bounds: $|\cld(\bm u)|\le2|q(\bm u)|^{-2}$ and $\lm\bigl(\cld(\bm u)\bigr)\le\pi|q(\bm u)|^{-4}$. To obtain lower bounds, let $T_{\bm u}$ be as in~\eqref{eq:Tu}.  A direct computation gives that $|T_{\bm u}'(z)|\asymp|q(\bm u)|^{-2}$ for $z\in\dmn_{\bm u}$.  Thus, for all Borel sets $A\subset\dmn_{\bm u}$,
\begin{equation}\label{eq:meaTA}
	\lm\bigl(T_{\bm u}(A)\bigr)=\int_A|T_{\bm u}'(z)|^2\D\lm(z) \asymp\lm(A)|q(\bm u)|^{-4}.
\end{equation}
In particular, $\lm\bigl(\cld(\bm u)\bigr)= \lm\bigl(T_{\bm u}(\dmn_{\bm u})\bigr)\asymp\lm\bigl(\dmn_{\bm u}\bigr)|q(\bm u)|^{-4}$. By Lemma~\ref{l:Ftype}, there are only finitely many different possibilities for~$\dmn_{\bm u}$, each of them is contained in $ \dmn $ and contains an open square with edge length $ 1/2 $. Consequently, $\lm\bigl(\cld(\bm u)\bigr)\asymp|q(\bm u)|^{-4}$, and so $|\cld(\bm u)|\asymp|q(\bm u)|^{-2}$. We summarize above in the lemma below.

\begin{lem}\label{l:diamea}
	There is a constant~$0<\cst_0<1$ such that, for all regular cylinders~$\cld(\bm u)$,
	\[ \frac{\cst_0}{|q(\bm u)|^2}\le|\cld(\bm u)|\le\frac{2}{|q(\bm u)|^2}\quad\text{and}\quad
	\frac{\pi\cst_0}{|q(\bm u)|^4}\le\lm\bigl(\cld(\bm u)\bigr)\le\frac{\pi}{|q(\bm u)|^4}. \]
\end{lem}

\subsection{On the boundary of regular cylinders} \label{ss:BRC}

In this subsection, we give two lemmas which concern rational points on the boundary of regular cylinders. This is the key to find complex irrational numbers which have quite different HCF partial quotients compared with their rational approximations. 

\begin{lem}\label{l:vk}
	For $k\ge0$, let
	\begin{align}
		\bm v_k&=(3,-2i,\underbrace{-2,2i,-2,-2i,\dots, -2,2i,-2,-2i}_{\text{$k$ repetitions of $-2,2i,-2,-2i$}}), \label{eq:vk} \\
		\tilde{\bm v}_k &=(3+i,2i,\underbrace{-2+i,2i,-2+i,2i,\dots,-2+i,2i,-2+i,2i}_{\text{$k$ repetitions of  $-2+i,2i,-2+i,2i$}}). \label{eq:vk'}
	\end{align}
	The following statements hold.
	\begin{enumerate}[\upshape(a)]
		\item \label{le:Fv} For all $k\ge0$, $\dmn_{\bm v_k}=\dmn_{-2i}=\dmn\setminus\overline B(i,1)$.
		\item \label{le:Fv'} For all $k\ge1$, $\dmn_{\tilde{\bm v}_k}=\dmn_{2i,-2+i,2i}=\dmn\cap\{z\in\C\colon|z+i|=1\}$.
		\item \label{le:v=v'} For all $k\ge0$, $[0;\bm v_k]=[0;\tilde{\bm v}_k]$.
	\end{enumerate}
\end{lem}

\begin{proof}
	(a) Example~\ref{e:-2i} gives that $\dmn_{-2i} =\dmn\setminus\overline B(i,1)$. We prove $\dmn_{\bm v_k}=\dmn_{-2i}$ by induction on~$k$. For $k=0$, we have $\cld(-2i)\subset\dmn=\dmn_3$, since $\cld(3)$ is full by Example~\ref{e:reg}. Hence, by Lemma~\ref{l:image}~\eqref{le:image}, $\dmn_{\bm v_0}=\dmn_{3,-2i}=\dmn_{-2i}$. Now assume that this holds for~$\bm v_{k-1}$. By Lemma~\ref{l:image}~\eqref{le:uv=u'v} and Example~\ref{e:-2i}, 
	\[ \dmn_{\bm v_k}=\dmn_{\bm v_{k-1},-2,2i,-2,-2i} =\dmn_{-2i,-2,2i,-2,-2i}=\dmn_{-2i}. \]
	
	(b) This can be proved in much the same way as~\eqref{le:Fv} through the replacement of Example~\ref{e:-2i} by Example~\ref{e:2i-}, and we therefore omit it.
	
	(c) Let $\bm v^*_k$ and $\tilde{\bm v}^*_k$ be the sequences obtained from $\bm v_k$ and $\tilde{\bm v}_k$ by deleting the first number, respectively. To simplify notation, write 
	\begin{align*}
		s_k&=p(\bm v^*_k),\quad t_k=q(\bm v^*_k),\quad
		s_k^-=p\bigl((\bm v^*_k)^-\bigr),\quad t_k^-=q\bigl((\bm v^*_k)^-\bigr),\\
		\tilde s_k&=p(\tilde{\bm v}^*_k),\quad \tilde t_k=q(\tilde{\bm v}^*_k),\quad
		\tilde s_k^-=p\bigl((\tilde{\bm v}^*_k)^-\bigr),\quad \tilde t_k^-=q\bigl((\tilde{\bm v}^*_k)^-\bigr).
	\end{align*}
	By~\eqref{eq:Qpair}, for $k\ge0$,
	\begin{align}
		\begin{pmatrix}
			s_k & s_k^-\\
			t_k & t_k^-
		\end{pmatrix} 
		&=
		\begin{pmatrix}
			0 & 1\\
			1 & 0
		\end{pmatrix}
		\begin{pmatrix}
			-2i & 1\\
			1 & 0
		\end{pmatrix}\left(
		\begin{pmatrix}
			-2 & 1 \\ 
			1 & 0	
		\end{pmatrix}
		\begin{pmatrix}
			2i & 1 \\ 
			1 & 0
		\end{pmatrix}
		\begin{pmatrix}
			-2 & 1 \\ 
			1 & 0
		\end{pmatrix}
		\begin{pmatrix}
			-2i & 1 \\ 
			1 & 0
		\end{pmatrix}
		\right)^k \notag\\
		&=\begin{pmatrix}
			1 & 0 \\ 
			-2i & 1
		\end{pmatrix}
		\begin{pmatrix}
			17+4i & -4+8i \\ 
			-8 & 1-4i
		\end{pmatrix}^k. \label{eq:stk}
	\end{align}
	Similarly, for $k\ge0$,
	\[\begin{pmatrix}
		\tilde{s}_k & \tilde{s}_k^-\\
		\tilde{t}_k & \tilde{t}_k^-
	\end{pmatrix}
	=
	\begin{pmatrix}
		1 & 0 \\ 
		2i & 1
	\end{pmatrix}
	\left(-\overline{
		\begin{pmatrix}
			17+4i & -4+8i \\ 
			-8 & 1-4i
	\end{pmatrix}}\right)^k=(-1)^k\overline{
		\begin{pmatrix}
			s_k & s_k^- \\ 
			t_k & t_k^-
	\end{pmatrix}}.\]
	Thus, $\tilde{s}_k/\tilde{t}_k=\overline{s_k/t_k}$ for $k\ge0$. Furthermore, we claim that
	\begin{equation}\label{eq:sktk}
		\Im(s_k/t_k)=1/2\quad\text{and}\quad
		\Re t_k=0\qquad\text{for $k\ge0$}.
	\end{equation}
	Consequently, $s_k/t_k=i+\tilde s_k/\tilde t_k$. Hence, for $k\ge0$,
	\[ [0;\bm v_k]=(3+s_k/t_k)^{-1}=(3+i+\tilde s_k/\tilde t_k)^{-1}=[0;\tilde{\bm v}_k]. \]
	
	It remains to prove~\eqref{eq:sktk}. We begin by proving
	\begin{equation}\label{eq:s=t-}
		s_k=t_k^-\qquad\text{for $k\ge0$}.
	\end{equation}
	For this, observe that the reverse sequence of~$\bm v^*_k$ is itself. By the mirror formula (Lemma~\ref{l:props}~\eqref{le:mf}),  
	\[ \frac{s_k}{t_k}=\frac{p(\bm v^*_k)}{q(\bm v^*_k)}=[0;\bm v^*_k]=\frac{q\bigl((\bm v^*_k)^-\bigr)}{q(\bm v^*_k)}=\frac{t_k^-}{t_k}. \]
	So we have $s_k=t_k^-$.
	
	We now prove~\eqref{eq:sktk} by induction on~$k$. For $k=0$, \eqref{eq:stk} gives that $s_0/t_0=i/2$ and $t_0=-2i$.	Now suppose that~\eqref{eq:sktk} holds for $k-1$. We first show that $\Re t_k=0$.  By~\eqref{eq:stk}, $t_k=(17+4i)t_{k-1}-8t_{k-1}^-$. Thus,
	\[ \frac{t_k}{t_{k-1}}=\frac{(17+4i)t_{k-1}-8t_{k-1}^-}{t_{k-1}}=17+4i-8\frac{s_{k-1}}{t_{k-1}}. \]
	Here we use~\eqref{eq:s=t-} to replace $t_{k-1}^-$ by~$s_{k-1}$. Consequently, by the inductive hypothesis $\Im(s_{k-1}/t_{k-1})=1/2$, we have $t_k/t_{k-1}\in\R$. This together with another inductive hypothesis $\Re t_{k-1}=0$ gives $\Re t_k=0$.
	
	Finally, we turn to prove $\Im(s_k/t_k)=1/2$. By~\eqref{eq:stk}, $s_k=(17+4i)s_{k-1}-8s_{k-1}^-$ and $t_k=(17+4i)t_{k-1}-8t_{k-1}^-$. This together with Lemma~\ref{l:props}~\eqref{le:pq-pq} gives
	\[ t_kt_{k-1}\left(\frac{s_k}{t_k}-\frac{s_{k-1}}{t_{k-1}}\right) =s_kt_{k-1}-t_ks_{k-1}=8(s_{k-1}t_{k-1}^--t_{k-1}s_{k-1}^-)\in\R. \]
	Note that $t_kt_{k-1}\in\R$ since $\Re t_k=\Re t_{k-1}=0$. It follows that $s_k/t_k-s_{k-1}/t_{k-1}\in\R$. Thus, $\Im(s_k/t_k)=\Im(s_{k-1}/t_{k-1})=1/2$.
\end{proof}

\begin{lem}\label{l:uvb}
	Let $\bm v_k$ and $\tilde{\bm v}_k$ be as in Lemma~\ref{l:vk}, $\bm a\in I^n$ with $\bm a$ being full and $b\in\Z[i]$ with $|b|\ge2\sqrt2$ and $\Im b\ge1$. Let $p=p(\bm a\bm v_k)$ and $q=q(\bm a\bm v_k)$. The following statements hold.
	\begin{enumerate}[\upshape(a)]
		\item \label{le:uvbF} For all $k\ge0$, the cylinder $\cld(\bm a\bm v_kb)$ is full and $\overline{\cld(\bm a\bm v_kb)}\subset\cld(\bm a)^\circ$.
		\item  \label{le:Cuvk'}  For $k\ge1$, $p/q\in\cld(\bm a\tilde{\bm v}_k)$.
		\item \label{le:uvdd} Let $\dd(\cdot,\cdot)$ be given by~\eqref{eq:dd}. Then for $k\ge1$,
		\[ \dd(z,p/q)=3k+2\qquad\text{for}\ z\in\cld(\bm a\bm v_kb). \]
	\end{enumerate}
\end{lem}
\begin{proof}
	(a) Let $T_b=(T|_{\cld(b)})^{-1}$, then $T_b(z)=(b+z)^{-1}$. Since $|b|\ge2\sqrt2$ and $\Im b\ge1$, one can check that
	\[ \overline\dmn+b\subset\bigl(\dmn^\circ\setminus\overline B(i,1)\bigr)^{-1}=\{z\in\C\colon\Im z>-1/2\} \setminus\bigcup_{z\in\{i,\pm1\}}\overline B(z,1), \]
	and so $\overline{T_b(\dmn)}\subset\bigl(\dmn\setminus\overline B(i,1)\bigr)^\circ=\dmn_{\bm v_k}^\circ$ by Lemma~\ref{l:vk}~\eqref{le:Fv}. Since $\cld(b)$ is full by Example~\ref{e:reg}, we have
	\begin{equation}\label{eq:b<vk}
		\overline{\cld(b)}=\overline{T_b(\dmn)}\subset\dmn_{\bm v_k}^\circ.
	\end{equation}
	Moreover, since $\bm a$ is full, we obtain by using Lemma~\ref{l:image}~\eqref{le:image} repeatedly that
	\[ \dmn_{\bm a\bm v_kb}=\dmn_{\bm v_kb}=\dmn_b=\dmn, \]
	which means that $\cld(\bm a\bm v_kb)$ is full.
	
	To prove $\overline{\cld(\bm a\bm v_kb)}\subset\cld(\bm a)^\circ$, let $T_{\bm a}$ and $T_{\bm v_k}$ be as in~\eqref{eq:Tu}. We have
	\[ \overline{\cld(\bm a\bm v_kb)}\subset T_{\bm a}\circ T_{\bm v_k}\overline{\cld(b)}\subset T_{\bm a}\circ T_{\bm v_k}(\dmn_{\bm v_k}^\circ)\subset T_{\bm a}(\dmn^\circ)=\cld(\bm a)^\circ. \]
	Here the second step follows from~\eqref{eq:b<vk}; the last from that $\cld(\bm a)$ is full.
	
	(b) Let $T_{\bm a\tilde{\bm v}_k}$ be as in~\eqref{eq:Tu}. We need three observations: $p/q=[0;\bm a\bm v_k]=[0;\bm a\tilde{\bm v}_k]$ (by Lemma~\ref{l:vk}~\eqref{le:v=v'}), $0\in\dmn_{\tilde{\bm v}_k}$ (by Lemma~\ref{l:vk}~\eqref{le:Fv'}) and $\dmn_{\tilde{\bm v}_k}=\dmn_{\bm a\tilde{\bm v}_k}$ (by Lemma~\ref{l:image}~\eqref{le:image} and the condition that $\bm a$ is full). It follows that
	\[ p/q=[0;\bm a\tilde{\bm v}_k]=T_{\bm a\tilde{\bm v}_k}(0)\in T_{\bm a\tilde{\bm v}_k}(\dmn_{\tilde{\bm v}_k})=T_{\bm a\tilde{\bm v}_k}(\dmn_{\bm a\tilde{\bm v}_k})=\cld(\bm a\tilde{\bm v}_k). \]
	
	(c) By~\eqref{le:Cuvk'}, $p/q\in\cld(\bm a\tilde{\bm v}_k)$. Hence, for $z\in\cld(\bm a\bm v_kb)$, $\dd(z,p/q)$ is exactly the number of different items of the two sequences $\bm a\bm v_k$ and $\bm a\tilde{\bm v}_k$, which is $3k+2$.
\end{proof}

\subsection{Auxiliary families of sequences} \label{ss:AF}

This subsection contains some useful lemmas which are needed in the proof of the lower bounds of dimensions of~$E(\psi)$.

\begin{lem}\label{l:I(r)}
	Given $0<r<1$, let
	\[ I(r)=\bigl\{b\in I\colon 1/r\le|b|\le2/r,\Im b\ge1\bigr\}. \]
	\begin{enumerate}[\upshape(a)]
		\item \label{le:CbSR} If $r\le\sqrt2/4$, then the cylinder~$\cld(b)$ is full for all $b\in I(r)$.
		\item \label{le:cardIr} $\lim_{r\to0}r^2\cdot\sharp I(r)=3\pi/2$. In particular, there exists a constant $r_0>0$ such that 
		\[ r^{-2}\le\sharp I(r)\le5r^{-2} \qquad\text{for all $r\le r_0$}. \]
	\end{enumerate}
\end{lem}
\begin{proof}
	(a) If $r\le\sqrt2/4$, then $|b|\ge1/r\ge2\sqrt2$ for all $b\in I(r)$. According to Example~\ref{e:reg}, we know that $\cld(b)$ is full.
	
	(b) This is an alternative version of the famous Gauss circle problem in number theory, which asserts that the number of lattice points $N(R)$ inside the boundary of a circle of radius~$R$ with center at the origin is $\pi R^2+O(R)$, see for instance~\cite[Theorem~339]{HarWr08}. Clearly, $\sharp I(r)=\bigl(N(2/r)-N(1/r)\bigr)/2+O(1/r)$.
\end{proof}

Given $M>0$, let $I_M=\{z\in I\colon|z|\le M\}$. 
\begin{lem}\label{l:mea<M}
	There exists a constant $\cst_1>1$ with the following property. For all $M>\sqrt{\cst_1}$ and $n\ge1$, we have
	\begin{equation}\label{eq:meaRM>}
		\sum_{\bm u\in I_M^n}\lm\bigl(\cld(\bm u)\bigr)\ge(1-\cst_1/M^2)^n.
	\end{equation}
\end{lem}
\begin{proof}
	We first prove that there exists a constant $\cst_1>1$ with the following property. For all sequences $\bm u$ (we allow $\bm u=\varnothing$) and all $M\ge1$,
	\begin{equation}\label{eq:meaub}
		\sum_{b\in I,|b|>M}\lm\bigl(\cld(\bm ub)\bigr)\le\cst_1\lm\bigl(\cld(\bm u)\bigr)/M^2.
	\end{equation}
	
	If $\bm u$ is irregular, Lemma~\ref{l:RCmeas} says $\lm\bigl(\cld(\bm u)\bigr)=0$. Both sides of~\eqref{eq:meaub} are equal to zero. Now suppose that $\bm u$ is regular, pick $z\in\cld\bigl(\bm ub\bigr)$, then $T^n(z)\in\cld(b)$. By Lemma~\ref{l:props}~\eqref{le:z-pq}, $|T^n(z)-1/b|\le|b|^{-2}$, and so $|T^n(z)|\le2/|b|<2/M$ for $|b|>M$. Hence, $\bigcup_{|b|>M}T^n\bigl(\cld(\bm ub)\bigr)\subset B(0,2/M)$. Let $T_{\bm u}$ be as in~\eqref{eq:Tu}, then
	\[ \bigcup_{|b|>M}\cld(\bm ub)\subset T_{\bm u}\bigl(B(0,2/M)\cap\dmn_{\bm u}\bigr). \]
	Combining this with~\eqref{eq:meaTA} and Lemma~\ref{l:diamea}, we obtain
	\[ \begin{split}
		\sum_{b\in I,|b|>M}\lm\bigl(\cld(\bm ub)\bigr) &\le \lm\bigl(T_{\bm u}(B(0,2/M)\cap\dmn_{\bm u})\bigr) \le c\frac{\lm\bigl(B(0,2/M)\cap\dmn_{\bm u})\bigr)}{|q(\bm u)|^4}\\ &\le c\frac{\lm\bigl(B(0,2/M))\bigr)}{|q(\bm u)|^4} \le\cst_1\frac{\lm\bigl(\cld(\bm u)\bigr)}{M^2}.
	\end{split} \]
	
	With this constant~$\cst_1$, we turn to prove~\eqref{eq:meaRM>} by induction on~$n$. For $n=0$, we adopt the convenience $I_M^0=\{\varnothing\}$ and $\cld(\varnothing)=\dmn$, then both sides of~\eqref{eq:meaRM>} are equal to~$1$. Now assume that this holds for $n-1$. By Lemma~\ref{l:RCmeas}, we have
	\begin{align*}
		\sum_{\bm u\in I_M^n}\lm\bigl(\cld(\bm u)\bigr) &= \sum_{\bm u\in I_M^{n-1}}\lm\bigl(\cld(\bm u)\bigr)-\sum_{\bm u\in I_M^{n-1}}\sum_{b\in I,|b|>M}\lm\bigl(\cld(\bm ub)\bigr)\\
		&\ge\sum_{\bm u\in I_M^{n-1}}\lm\bigl(\cld(\bm u)\bigr)-\sum_{\bm u\in I_M^{n-1}}\cst_1\lm\bigl(\cld(\bm u)\bigr)/M^2 \qquad(\text{by~\eqref{eq:meaub}})\\
		&=(1-\cst_1/M^2)\sum_{\bm u\in I_M^{n-1}}\lm\bigl(\cld(\bm u)\bigr)=(1-\cst_1/M^2)^n.
	\end{align*}
	Here we use the inductive hypothesis in the last equality.
\end{proof}

\begin{lem}\label{l:OMQ}
	Given $M,Q>0$, let $\cst_1>1$ be the constant in Lemma~\ref{l:mea<M} and
	\[ \Gamma_M(Q)=\biggl\{\bm u\in\bigcup_{j\ge1} I_M^j\colon\text{$\bm u$ is full and $|q(\bm u^-)|<Q\le |q(\bm u)|$}\biggr\}. \]
	If $M\ge12\cst_1$ and $Q\ge(M+1)^{5M}$, then $\sharp\Gamma_M(Q)\ge Q^{4-2/M}$.
\end{lem}
\begin{proof}
	The idea is to count the number of prefix of sequences in~$\Gamma_M(Q)$ instead of counting them directly. Since the prefix of a full sequence may be regular, let
	\[ \Gamma_M^-(Q)=\biggl\{\bm u\in\bigcup_{j\ge1} I_M^j\colon\text{$\bm u$ is \emph{regular} and $|q(\bm u^-)|<\frac{Q}{M+1}\le |q(\bm u)|$}\biggr\}. \]
	We claim that
	\begin{enumerate}[\upshape(a)]
		\item $\sharp\Gamma_M(Q)\ge\sharp\Gamma_M^-(Q)$.
		\item Let $n=\lfloor6\log Q\rfloor$. If $\bm v\in I_M^n$ is regular, then $\cld(\bm v)\subset\bigcup_{\bm u\in\Gamma_M^-(Q)}\cld(\bm u)$.
	\end{enumerate}
	
	Since every irregular cylinder has Lebesgue measure zero (by Lemma~\ref{l:RCmeas}), it follows from Claim~(b) and Lemma~\ref{l:mea<M} that
	\begin{align*}
		\lm\biggl(\bigcup_{\bm u\in\Gamma_M^-(Q)}\cld(\bm u)\biggr) &\ge\lm\biggl(\bigcup_{\bm u\in I_M^n}\cld(\bm u)\biggr) \ge(1-\cst_1/M^2)^{6\log Q}\\ 
		&=Q^{6\log(1-\cst_1/M^2)} >Q^{-12\cst_1/M^2}\ge Q^{-1/M}.
	\end{align*}
	The last two inequalities follow from the fact $\log(1-x)>-2x$ if $0<x<1/2$ and the condition $M\ge12\cst_1$. On the other hand, Lemma~\ref{l:diamea} gives that, for all $\bm u\in\Gamma_M^-(Q)$,
	\[ \lm\bigl(\cld(\bm u)\bigr) \le\pi|q(\bm u)|^{-4} \le\pi(M+1)^4/Q^4 <Q^{-4+1/M}. \]
	The last inequality follows from $\log_Q\pi(M+1)^4<5\log_Q(M+1)\le1/M$. The two bounds above together with Claim~(a) gives
	\[ \sharp\Gamma_M(Q)\ge\sharp\Gamma_M^-(Q) \ge\frac{\lm\bigl(\bigcup_{\bm u\in\Gamma_M^-(Q)}\cld(\bm u)\bigr)}{\max_{\bm u\in\Gamma_M^-(Q)}\lm\bigl(\cld(\bm u)\bigr)} \ge\frac{Q^{-1/M}}{Q^{-4+1/M}}=Q^{4-2/M}. \]
	
	It remains to prove the two claims. For Claim~(a), given $\bm u\in\Gamma_M^-(Q)$, we shall prove that $\bm u$ is a prefix of some $\bm u'\in\Gamma_M(Q)$. Since $\bm u$ is regular, by Lemma~\ref{l:newRC}~\eqref{le:ubF}, \eqref{le:uvsr} and Example~\ref{e:reg}, there is $b\in\{\pm2+\pm2i\}$ with $|b|\le M$ such that the sequence 
	\[ \bm ub^{*k}:=(\bm u,\underbrace{b,\dots,b}_k)\ \text{is full for all $k\ge1$.} \]
	By Lemma~\ref{l:props}~\eqref{le:|q-|}, $|q(\bm u)|<(M+1)|q(\bm u^-)|<Q$. Combining this and Lemma~\ref{l:props}~\eqref{le:1<q}, one sees that there exists $k\ge1$ such that $|q(\bm ub^{*(k-1)})|<Q\le|q(\bm ub^{*k})|$. Hence, $\bm u'=\bm ub^{*k}\in\Gamma_M(Q)$. Moreover, distinct $\bm u,\bm v\in\Gamma_M^-(Q)$ yield distinct $\bm u',\bm v'\in\Gamma_M(Q)$, since 
	\[ \cld(\bm u')\cap\cld(\bm v')\subset\cld(\bm u)\cap\cld(\bm v)=\emptyset. \]
	This completes the proof of Claim~(a).
	
	For Claim~(b), given $\bm v=(v_1,\dots,v_n)\in I_M^n$ being regular, we shall prove that $\cld(\bm v)\subset\cld(\bm u)$ for some $\bm u\in\Gamma_M^-(Q)$. By Lemma~\ref{l:props}~\eqref{le:q>q},
	\[ |q(\bm v)|\ge\phi^{\lfloor n/2\rfloor} \ge\phi^{(6\log Q-3)/2} \ge\phi^{(5\log Q)/2} >e^{\log Q}=Q. \]
	By Lemma~\ref{l:props}~\eqref{le:1<q}, there is $1<k\le n$ such that
	\[ |q(v_1,\dots,v_{k-1})|<Q/(M+1)\le|q(v_1,\dots,v_k)|. \]
	By Lemma~\ref{l:newRC}~\eqref{le:ur}, $\bm u=(v_1,\dots,v_k)$ is regular, and so $\bm u\in\Gamma_M^-(Q)$. Since $\cld(\bm v)\subset\cld(\bm u)$, Claim~(b) is proved.
\end{proof}

Let $\cst_1>1$ be the constant in Lemma~\ref{l:mea<M} and $\Gamma_M(Q)$ as in Lemma~\ref{l:OMQ}. 
\begin{lem}\label{l:CaQ}
	Suppose that $M\ge12\cst_1$ and $Q>0$. Given a sequence $\bm w$, let
	\[ \Gamma_M^{\bm w}(Q)=\bigl\{\bm b\colon\bm{wb}\in\Gamma_M(Q)\bigr\}. \]
	Then
	\[ \sharp\Gamma_M^{\bm w}(Q)\le(M+1)^{24M}|q(\bm w)|^{-4+2/M}\cdot \sharp\Gamma_M(Q). \]
\end{lem}
\begin{proof}
	We can assume that $|q(\bm w)|>(M+1)^{6M}$, for otherwise
	\[ \sharp\Gamma_M^{\bm w}(Q)\le\sharp\Gamma_M(Q) <(M+1)^{24M}|q(\bm w)|^{-4+2/M}\cdot \sharp\Gamma_M(Q). \]
	Write $Q'=\frac{|q(\bm w)|}{15(M+1)^2}>\frac{(M+1)^{6M}}{15(M+1)^2}>(M+1)^{5M}$, Lemma~\ref{l:OMQ} gives
	\begin{equation}\label{eq:Q'/Q}
		\sharp\Gamma_M(Q') \ge\Bigl(\frac{|q(\bm w)|}{15(M+1)^2}\Bigr)^{4-2/M} >\frac{|q(\bm w)|^{4-2/M}}{(M+1)^{24M}}.
	\end{equation}
	We claim that
	\begin{equation}\label{eq:Q'Q}
		\sharp\Gamma_M^{\bm w}(Q)\cdot\sharp\Gamma_M(Q') \le\sharp\Gamma_M(Q).
	\end{equation}
	Clearly, the inequality~\eqref{eq:Q'Q} together with~\eqref{eq:Q'/Q} completes the proof.
	
	It remains to prove~\eqref{eq:Q'Q}. Assume that the left side of~\eqref{eq:Q'Q} is greater than zero, for otherwise it is trivial. This enables us to pick $\bm b\in\Gamma_M^{\bm w}(Q)$ and $\bm c\in\Gamma_M(Q')$. By Lemma~\ref{l:newRC}~\eqref{le:usr}, $\bm b$ is full. By the definition of~$\Gamma_M(Q')$, so is~$\bm c$. By Lemma~\ref{l:newRC}~\eqref{le:uvsr} and Example~\ref{e:reg}, the sequence
	\begin{equation}\label{eq:bc3}
		\bm b\bm c3^{*k}:=(\bm b\bm c,\underbrace{3,\dots,3}_k)\ \text{is full for all $k\ge1$}.
	\end{equation}
	
	By Lemma~\ref{l:props}~\eqref{le:|q-|}, \eqref{le:|qq|} and the definition of~$\Gamma_M^{\bm w}(Q)$,
	\[ |q(\bm b)|<\frac{5|q(\bm{wb})|}{|q(\bm w)|}<\frac{5(M+1)|q(\bm{wb}^-)|}{|q(\bm w)|}<\frac{5(M+1)Q}{|q(\bm w)|}. \]
	By Lemma~\ref{l:props}~\eqref{le:|q-|} and the definition of~$\Gamma_M(Q')$, 
	\[ |q(\bm c)|<(M+1)|q(\bm c^-)|<(M+1)Q'=\frac{|q(\bm w)|}{15(M+1)}. \]
	Combing the two upper bounds and Lemma~\ref{l:props}~\eqref{le:|qq|}, we have 
	\[ |q(\bm{bc})|<3|q(\bm b)q(\bm c)|<Q. \]
	By Lemma~\ref{l:props}~\eqref{le:1<q}, there exists $k\ge1$ such that
	\[ |q(\bm{bc}3^{*(k-1)})|<Q\le|q(\bm{bc}3^{*k})|. \]
	This together with~\eqref{eq:bc3} gives $\bm{bc}3^{*k}\in\Gamma_M(Q)$. Moreover, distinct $(\bm b,\bm c)$, $(\bm b',\bm c')$ yield distinct $\bm b\bm c3^{*k},\bm b'\bm c'3^{*k'}\in\Gamma_M(Q)$, since
	\[ \cld(\bm b\bm c3^{*k})\cap\cld(\bm b'\bm c'3^{*k'}) \subset\cld(\bm b\bm c)\cap\cld(\bm b'\bm c')=\emptyset. \]
	Therefore, the inequality~\eqref{eq:Q'Q} holds.
\end{proof}

\section{Proof of Theorems~\ref{t:result} and~\ref{t:Wpsi}} \label{sec:proof}

It is worth noting that the case $\lambda(\psi)\le2$ can be derived from the case $\lambda(\psi)>2$. To see this, assuming $\lambda(\psi)\le2$, we consider the function $\psi_\epsilon(x):=x^{\lambda(\psi)-2-\epsilon}\psi(x)$ instead of~$\psi(x)$. For all $\epsilon>0$, the function $\psi_\epsilon$ is also nonincreasing and $\lambda(\psi_\epsilon)=2+\epsilon>2$. Moreover, we have $E(\psi_\epsilon)\subset E(\psi)$, since $\psi_\epsilon(x)<\psi(x)$ for $x>1$. If Theorem~\ref{t:result} holds for $\psi_\epsilon$, it follows that, for all $\epsilon>0$,
\[ \hdim E(\psi)\ge\hdim E(\psi_\epsilon)=4/(2+\epsilon). \]
Letting $\epsilon\to0$, we have $\pdim W(\psi)=\hdim W(\psi)=\pdim E(\psi)=\hdim E(\psi)=2$ since $E(\psi)\subset W(\psi)\subset\C$. This completes the proof of Theorems~\ref{t:result} and~\ref{t:Wpsi} for the case $\lambda(\psi)\le2$.

From now on, fix $\psi$ with $\lambda(\psi)>2$ and denote $\lambda(\psi)$ briefly by~$\lambda$. We will prove Theorems~\ref{t:result} and~\ref{t:Wpsi} by obtaining the upper bounds for dimensions of~$W(\psi)$ (see Section~\ref{ss:UB}) and lower bounds for dimensions of~$E(\psi)$ (see Section~\ref{ss:LB}).

\subsection{Upper bounds for dimensions of $W(\psi)$} \label{ss:UB}

Clearly, $\pdim W(\psi)\le2$. 

For Hausdorff dimension, let $\|z\|=\max\{|\Re z|,|\Im z|\}$. One can check that
\[ W(\psi)\subset\bigcap_{n\ge1}\bigcup_{k\ge n}\bigcup_{\substack{p,q\in\Z[i]\\ \|p\|\le\|q\|=k}}\overline B(p/q,\psi(|q|)). \]
Note that, for $k\ge1$,
\[ \sharp\{(p,q)\in\Z[i]\colon\|p\|\le\|q\|=k\}=8k(2k+1)^2, \] 
since $ \sharp\{q\in\Z[i]\colon\|q\|=k\}=8k $ and $ \sharp\{p\in\Z[i]\colon\|p\|\le k\}=(2k+1)^2 $.
Pick $\epsilon>0$ sufficiently small, let $s=\frac{4+\epsilon}{\lambda-\epsilon}$. Since $\psi(|q|)\le|q|^{-\lambda+\epsilon}\le(\|q\|)^{-\lambda+\epsilon}$ when $|q|$ large enough, we have
\[ \hdm^s\bigl(W(\psi)\bigr) \le\sum_{\text{$k$ large enough}}\sum_{\substack{p,q\in\Z[i]\\ \|p\|\le\|q\|=k}} \bigl(2\psi(|q|)\bigr)^s \le c\cdot\sum_{k\ge 1}k^3 (k^{-\lambda+\epsilon})^{\frac{4+\epsilon}{\lambda-\epsilon}}<\infty, \]
where $c>0$ is a constant. This implies $\hdim W(\psi)\le\frac{4+\epsilon}{\lambda-\epsilon}$. Let $\epsilon\to0$, we obtain
\begin{equation}\label{eq:upbd}
	\hdim W(\psi)\le4/\lambda.
\end{equation} 

\subsection{Auxiliary parameters and geometric bounds} \label{ss:para}

Recall that we assume $\lambda>2$. Fix $\epsilon>0$ sufficiently small (e.g., $100\epsilon<\min(\lambda-2,\lambda^{-1})$). Fix $M=\max(12\cst_1,2\epsilon^{-1})$, where $\cst_1>1$ is the constant in Lemma~\ref{l:mea<M}. Let $(n_k)_{k\ge0}$ be a rapidly increasing sequence of positive numbers with $n_0=1$, which satisfies
\begin{align}
	&n_k^{-\lambda-\epsilon/4}\le\psi(n_k)\le n_k^{-\lambda+\epsilon/4}  \mspace{12mu}\qquad\text{for $k\ge1$}, \label{eq:nkpsi}\\
	&n_k^{\epsilon/4}\ge 27(M+1)|q(\bm v_k)|n_{k-1}^\lambda  
	\qquad\text{for $k\ge1$}, \label{eq:nk>}\\
	&n_1\ge\max\bigl((M+1)^{5M/(1-\epsilon/4)},r_0^{-(\lambda-2+\epsilon)}\bigr), \label{eq:n1}
\end{align}
where $\bm v_k$ are the sequences given by~\eqref{eq:vk} and $r_0$ is the constant in Lemma~\ref{l:I(r)}~\eqref{le:cardIr}.

For $k\ge1$, set $Q_k=n_k^{1-\epsilon/4}$ and $r_k=n_k^{-(\lambda-2+\epsilon)}$. Recall the definition of~$I(r)$ and~$\Gamma_M(Q)$ from Lemmas~\ref{l:I(r)} and~\ref{l:OMQ}, respectively. Since $(n_k)$ is increasing, it follows from~\eqref{eq:n1} and Lemma~\ref{l:I(r)}~\eqref{le:cardIr} that, for $k\ge1$,
\begin{equation}\label{eq:QI(r)}
	Q_k\ge(M+1)^{5M}\quad\text{and}\quad 5n_k^{2\lambda-4+2\epsilon}\ge\sharp I(r_k)\ge n_k^{2\lambda-4+2\epsilon}.
\end{equation}
Recall that $\bm v_k$ are the sequences given by~\eqref{eq:vk}. Let $\Lambda_0=\Lambda'_0=\{\varnothing\}$. For $k\ge1$, define~$\Lambda_k$ and~$\Lambda_k'$ by
\begin{equation}\label{eq:Lam}
	\begin{aligned}
		\Lambda_k&=\prod_{j=1}^k\bigl(\Gamma_M(Q_j)\times\{\bm v_jb\colon b\in I(r_j)\}\bigr),\\
		\Lambda'_k&=\Lambda_{k-1}\times\Gamma_M(Q_k).
	\end{aligned}
\end{equation}

\begin{lem}\label{l:Lam}
	Let $\cst_0$ be the constant in Lemma~\ref{l:diamea}. For $k\ge1$, the following statements hold.
	\begin{enumerate}[\upshape(a)]
		\item \label{le:LamF} All the sequences in $\Lambda_k$ and $\Lambda_k'$ are full.
		\item \label{le:NLam} $\sharp\Lambda_k\ge n_k^{2\lambda}$ and $\sharp\Lambda'_k\ge n_k^{4-2\epsilon}$.
		\item \label{le:qLam} For $\bm a\in\Lambda_k$, $n_k^{\lambda-1} \le|q(\bm a)| \le n_k^{\lambda-1+\epsilon}$. Moreover,
		\[ \cst_0n_k^{-2\lambda+2-2\epsilon}\le|\cld(\bm a)|\le2n_k^{-2\lambda+2},\quad \lm\bigl(\cld(\bm a)\bigr)\ge\pi\cst_0n_k^{-4\lambda+4-4\epsilon}. \]
		\item \label{le:qLam'} For $\bm a'\in\Lambda_k'$, $n_k^{1-\epsilon/3}\le|q(\bm a')|<|q(\bm a'\bm v_k)|\le n_k/3$. Moreover,
		\[ |\cld(\bm a')|\le2n_k^{-2+\epsilon},\quad \lm\bigl(\cld(\bm a')\bigr)\ge\pi\cst_0n_k^{-4}. \]
		\item \label{le:dd<psi} Let $\bm a\in\Lambda_k$, if $z\in\cld(\bm a)$, then $|z-p(\bm a^-)/q(\bm a^-)|\le\psi(|q(\bm a^-)|)$.
		\item \label{le:dtcld} For $\bm a_1,\bm a_2\in\Lambda_k$ with $\bm a_1^-\ne\bm a_2^-$, the Euclidean distance
		\[ \dist\bigl(\cld(\bm a_1),\cld(\bm a_2)\bigr)\ge n_k^{-2}. \]
	\end{enumerate}
\end{lem}
\begin{proof}
	(a) This follows from Lemma~\ref{l:newRC}~\eqref{le:uvsr} and Lemma~\ref{l:uvb}~\eqref{le:uvbF}.
	
	(b) By Lemma~\ref{l:OMQ}, \eqref{eq:QI(r)} and~$2/M\le\epsilon$,
	\[ \begin{split}
		\sharp\Lambda_k&=\sharp\Lambda_{k-1}\cdot\sharp\Gamma_M(Q_k) \cdot\sharp I(r_k)\ge\sharp\Gamma_M(Q_k)\cdot\sharp I(r_k)\\
		&\quad\ge Q_k^{4-2/M}\cdot n_k^{2\lambda-4+2\epsilon}\ge(n_k^{1-\epsilon/4})^{4-\epsilon}\cdot n_k^{2\lambda-4+2\epsilon}>n_k^{2\lambda}.\\
		\sharp\Lambda'_k&=\sharp\Lambda_{k-1} \cdot\sharp\Gamma_M(Q_k) \ge\sharp\Gamma_M(Q_k) \ge (n_k^{1-\epsilon/4})^{4-\epsilon} >n_k^{4-2\epsilon}.
	\end{split} \]
	
	(c) Write $\bm a=\tilde{\bm a}\bm u\bm v_kb$ with $\tilde{\bm a}\in\Lambda_{k-1}$, $\bm u\in\Gamma_M(Q_k)$ and $b\in I(r_k)$. By Lemma~\ref{l:props}~\eqref{le:|qq|} and~\eqref{eq:nk>},
	\[ |q(\bm a)|=|q(\tilde{\bm a}\bm u\bm v_kb)|>|q(\bm u)q(b)|/125\ge n_k^{1-\epsilon/4}\cdot n_k^{\lambda-2+\epsilon}/125>n_k^{\lambda-1}. \]
	We prove the upper bound by induction on~$k$. For $k=0$, we have $|q(\varnothing)|=1=n_0^{\lambda-1+\epsilon}$. Suppose this holds for $k-1$, by Lemma~\ref{l:props}~\eqref{le:|q-|}, \eqref{le:|qq|} and~\eqref{eq:nk>},
	\[ \begin{split}
		|q(\bm a)|&=|q(\tilde{\bm a}\bm u\bm v_kb)|\le 9|q(\tilde{\bm a})q(\bm u)q(\bm v_k)|(|b|+1)\\
		&<9n_{k-1}^{\lambda-1+\epsilon}\cdot(M+1)n_k^{1-\epsilon/4}\cdot|q(\bm v_k)|\cdot 3n_k^{\lambda-2+\epsilon}\\
		&\le 27(M+1)|q(\bm v_k)|n_{k-1}^\lambda\cdot n_k^{-\epsilon/4}\cdot n_k^{\lambda-1+\epsilon}\le n_k^{\lambda-1+\epsilon}.
	\end{split} \]
	The bounds for diameter and Lebesgue measure follow from Lemma~\ref{l:diamea}.
	
	(d) Lemma~\ref{l:props}~\eqref{le:1<q} gives $|q(\bm a')|<|q(\bm a'\bm u)|$. Note that \eqref{eq:nk>} implies $n_k^{\epsilon/4}>6^3$ since $M>12$. Write $\bm a'=\bm a\bm u$ with $\bm a\in\Lambda_{k-1}$ and $\bm u\in\Gamma_M(Q_k)$. By Lemma~\ref{l:props}~\eqref{le:|qq|},
	\[ |q(\bm a')|=|q(\bm a\bm u)|>|q(\bm a)q(\bm u)|/5 \ge|q(\bm u)|/5\ge n_k^{1-\epsilon/4}/5>n_k^{1-\epsilon/3}. \]
	By~\eqref{le:qLam}, Lemma~\ref{l:props}~\eqref{le:|q-|}, \eqref{le:|qq|} and~\eqref{eq:nk>},
	\[ \begin{split}
		|q(\bm a'\bm v_k)|&=|q(\bm a\bm u\bm v_k)|< 9|q(\bm a)q(\bm u)q(\bm v_k)| \le 9n_{k-1}^{\lambda-1+\epsilon}\cdot(M+1)n_k^{1-\epsilon/4}\cdot|q(\bm v_k)|\\
		&\le 27(M+1)|q(\bm v_k)|n_{k-1}^\lambda\cdot n_k^{-\epsilon/4}\cdot (n_k/3)\le n_k/3.
	\end{split} \]
	The bounds for diameter and Lebesgue measure follow from Lemma~\ref{l:diamea}.
	
	(e) Write $\bm a=\bm a'\bm v_kb$ with $\bm a'\in\Lambda_k'$ and $b\in I(r_k)$. For $z\in\cld(\bm a)$, let $n$ be the length of~$\bm a^-$, then $p(\bm a^-)=p_n$, $q(\bm a'\bm v_k)=q(\bm a^-)=q_n$ and $b=a_{n+1}$. By Lemma~\ref{l:props}~\eqref{le:z-pq}, 
	\[ \begin{split}
		\left|z-\frac{p(\bm a^-)}{q(\bm a^-)}\right|&=\left|z-\frac{p_n}{q_n}\right|=\frac{1}{|q_n|^2|b+T^{n+1}(z)+q_{n-1}/q_n|}<\frac{1}{(|b|-2)|q(\bm a'\bm v_k)|^2}\\
		&<\frac{2}{n_k^{\lambda-2+\epsilon}\cdot n_k^{2(1-\epsilon/3)}} <n_k^{-\lambda-\epsilon/4}\le\psi(n_k)<\psi(|q(\bm a^-)|).
	\end{split} \]
	Here the first inequality follows from $|T^{n+1}(z)|<1$ (by $T^{n+1}(z)\in\dmn$) and $|q_{n-1}|<|q_n|$ (by Lemma~\ref{l:props}~\eqref{le:1<q}); the second from the lower bound for $|q(\bm a'\bm v_k)|$ in~\eqref{le:qLam'}; the fourth from~\eqref{eq:nkpsi}; the last from the upper bound for $|q(\bm a'\bm v_k)|$ in~\eqref{le:qLam'} and the fact that $\psi$ is nonincreasing.
	
	(f) Write $\bm a_1=\bm a_1'\bm v_kb_1$ and $\bm a_2=\bm a_2'\bm v_kb_2$ with $\bm a_1',\bm a_2'\in\Lambda_k'$ and $b_1,b_2\in I(r_k)$. Then $\bm a_1^-=\bm a_1'\bm v_k$ and $\bm a_2^-=\bm a_2'\bm v_k$. By~\eqref{le:LamF}, both $\bm a_1'$ and $\bm a_2'$ are full. Hence, from Lemma~\ref{l:uvb}~\eqref{le:Cuvk'},
	\[ p(\bm a_1^-)/q(\bm a_1^-)\in\cld(\bm a_1'\tilde{\bm v}_k)\quad\text{and}\quad
	p(\bm a_2^-)/q(\bm a_2^-)\in\cld(\bm a_2'\tilde{\bm v}_k). \]
	The condition $\bm a_1^-\ne\bm a_2^-$ implies $\bm a_1'\ne\bm a_2'$, and so $p(\bm a_1^-)/q(\bm a_1^-)\ne p(\bm a_2^-)/q(\bm a_2^-)$, since $\cld(\bm a_1'\tilde{\bm v}_k)\cap\cld(\bm a_2'\tilde{\bm v}_k)=\emptyset$. Furthermore, since $p(\bm a_1^-)$, $q(\bm a_1^-)$, $p(\bm a_2^-)$ and $q(\bm a_2^-)$ are all Gaussian integers, we have
	\begin{equation}\label{eq:pq-pq>}
		\left|\frac{p(\bm a_1^-)}{q(\bm a_1^-)}-\frac{p(\bm a_2^-)}{q(\bm a_2^-)}\right| \ge\left|\frac{1}{q(\bm a_1^-)q(\bm a_2^-)}\right| \ge\frac{9}{n_k^2}.
	\end{equation}
	The last inequality follows from the fact that both $|q(\bm a_1^-)|$ and $|q(\bm a_2^-)|$ are less than~$n_k/3$ (by the upper bound for $|q(\bm a'\bm v_k)|$ in~\eqref{le:qLam'}).
	
	Pick $z_1\in\cld(\bm a_1)$ and $z_2\in\cld(\bm a_2)$. According to the proof of~\eqref{le:dd<psi},
	\[ |z_1-p(\bm a_1^-)/q(\bm a_1^-)|\le n_k^{-\lambda-\epsilon/4}\le n_k^{-2}. \]
	Similarly, $|z_2-p(\bm a_2^-)/q(\bm a_2^-)|\le n_k^{-2}$. Combining these inequalities and~\eqref{eq:pq-pq>} gives $|z_1-z_2|\ge7/n_k^2>n_k^{-2}$. This completes the proof.
\end{proof}

\subsection{Lower bounds for dimensions of $E(\psi)$} \label{ss:LB}

In this subsection, we will construct a subset $E\subset E(\psi)$ and then obtain the lower bounds of its dimensions. 

We follow the notation used in Section~\ref{ss:para}. For the definition of~$I(r)$ see Lemma~\ref{l:I(r)}; for $\Gamma_M(Q)$ see Lemma~\ref{l:OMQ}; for $\Lambda_k$ and $\Lambda_k'$ see~\eqref{eq:Lam}; for $\bm v_k$ see~\eqref{eq:vk}.

Let $\Omega_k=\Gamma_M(Q_k)\times\{\bm v_kb\colon b\in I(r_k)\}$ for $k\ge1$. Define $\Omega=\prod_{k\ge1}\Omega_k$ and
\begin{equation}\label{eq:E}
	E=\bigcap_{k\ge1}\bigcup_{\bm a\in\Lambda_k}\cld(\bm a).
\end{equation}
Let $\omega=\bm w_1\bm w_2\dotsc\in\Omega$ with 
\begin{equation}\label{eq:wk}
	\bm w_k=\bm u_k\bm v_kb_k\in\Omega_k=\Gamma_M(Q_k)\times\{\bm v_kb\colon b\in I(r_k)\}.
\end{equation}
By Lemma~\ref{l:Lam}~\eqref{le:LamF} and Lemma~\ref{l:uvb}~\eqref{le:uvbF},
\[ \overline{\cld(\bm w_1\ldots\bm w_{k-1}\bm w_k)} \subset\cld(\bm w_1\ldots\bm w_{k-1}\bm u_k)^\circ \subset\cld(\bm w_1\ldots\bm w_{k-1})^\circ. \]
So $\bigcap_{k\ge1}\cld(\bm w_1\ldots\bm w_k)\ne\emptyset$. In fact, it contains exactly one point $z=[0;\omega]$. Clearly, $\omega$ is the sequence of HCF partial quotients of~$z$. Note that $z\in E$ since $\bm w_1\ldots\bm w_k\in\Lambda_k$ for all $k\ge1$. Thus, $\Omega$ is just the set consisting of the sequences of HCF partial quotients of all points in~$E$. This induces a one-to-one map
\[ \Phi\colon\Omega\to E,\quad\omega\mapsto z=[0;\omega]. \]

Let $\nu_k$ be the counting measures on the set~$\Omega_k$. Define $\nu=\prod_{k\ge1}(\nu_k/\sharp\Omega_k)$ to be the product measure on~$\Omega$, and $\mu=\nu\circ\Phi^{-1}$ the projection measure of~$\nu$ under the map~$\Phi$. By the definition, $\mu$ is the unique probability measure supported on~$E$ such that, for $k\ge1$,
\begin{equation}\label{eq:mu}
	\begin{aligned}
		\mu\bigl(\cld(\bm a)\bigr)&=(\sharp\Lambda_k)^{-1} \qquad\text{for all $\bm a\in\Lambda_k$},\\
		\mu\bigl(\cld(\bm a')\bigr)&=(\sharp\Lambda_k')^{-1} \qquad\text{for all $\bm a'\in\Lambda_k'$}.
	\end{aligned}
\end{equation}

\begin{lem}\label{l:E}
	$E\subset E(\psi)$.
\end{lem}
\begin{proof}
	Pick $z\in E$, let $\omega\in\Omega$ be its sequence of HCF partial quotients of the form $\omega=\bm w_1\bm w_2\dotsc$ with $\bm w_k=\bm u_k\bm v_kb_k$ given by~\eqref{eq:wk}. Let $\bm a_0=\varnothing$ and for $k\ge1$,
	\[ \bm a_k=\bm w_1\ldots\bm w_k\in\Lambda_k,\quad \bm a_k'=\bm a_{k-1}\bm u_k\in\Lambda_k'. \]
	then $\bm a_k=\bm a_k'\bm v_kb_k$ and $\bm a_k^-=\bm a_k'\bm v_k$. Clearly, $z\in\cld(\bm a_k)$. By Lemma~\ref{l:Lam}~\eqref{le:LamF}, $\bm a_k'$ is full. It follows from Lemma~\ref{l:uvb}~\eqref{le:uvdd} and Lemma~\ref{l:Lam}~\eqref{le:dd<psi} that, for all $k\ge1$,
	\[ \dd\bigl(z,p(\bm a_k^-)/q(\bm a_k^-)\bigr)=3k+2,\quad
	|z-p(\bm a_k^-)/q(\bm a_k^-)|\le\psi(|q(\bm a_k^-)|). \]
	Hence, $z\in E(\psi)$, and so $E\subset E(\psi)$.
\end{proof}

Let $(n_k)_{k\ge0}$ be the sequence given in~Section~\ref{ss:para}. 
\begin{lem}\label{l:meaB}
	There exists a constant~$\cst$ with the following property. Let $z\in E$ and $0<r<n_1^{-2}$. Suppose that $n_{k+1}^{-2}\le r<n_k^{-2}$ for some $k\ge1$, then
	\begin{enumerate}[\upshape(a)]
		\item \label{le:rnk+1} $\mu\bigl(B(z,r)\bigr)\le\cst r^{2-2\epsilon}$ for $n_{k+1}^{-2}\le r<n_{k+1}^{-2+\epsilon}$;
		\item $\mu\bigl(B(z,r)\bigr)\le\cst r^{\lambda/(\lambda-1)-3\epsilon}$ for $n_{k+1}^{-2+\epsilon}\le r<n_k^{-2\lambda+2-2\epsilon}$;
		\item $\mu\bigl(B(z,r)\bigr)\le\cst r^{4/\lambda-4\epsilon}$ for $n_k^{-2\lambda+2-2\epsilon}\le r<n_k^{-2}$.
	\end{enumerate}
\end{lem}
\begin{proof}
	(a) For $z\in E$, let 
	\[ \Lambda_{k+1}'(z)=\bigl\{\bm a'\in\Lambda_{k+1}'\colon\cld(\bm a')\cap B(z,r)\ne\emptyset\bigr\}. \]
	By Lemma~\ref{l:Lam}~\eqref{le:qLam'}, $|\cld(\bm a')|\le2n_{k+1}^{-2+\epsilon}$ for $\bm a'\in\Lambda_{k+1}'$. Combining this and the condition $r<n_{k+1}^{-2+\epsilon}$ gives
	\[ \bm a'\in\Lambda_{k+1}'(z)\implies \cld(a')\subset B(z,3n_{k+1}^{-2+\epsilon}). \]
	Hence, by the lower bound for $\lm\bigl(\cld(\bm a')\bigr)$ in Lemma~\ref{l:Lam}~\eqref{le:qLam'},
	\[ \sharp\Lambda_{k+1}'(z)\le\frac{\lm(B(z,3n_{k+1}^{-2+\epsilon}))}{\min_{\bm a'\in\Lambda_{k+1}'}\lm(\cld(\bm a'))} \le\frac{9\pi n_{k+1}^{-4+2\epsilon}}{\pi\cst_0n_{k+1}^{-4}} =9\cst_0^{-1}n_{k+1}^{2\epsilon}. \]
	It follows from Lemma~\ref{l:Lam}~\eqref{le:NLam} and~\eqref{eq:mu} that
	\[ \mu\bigl(B(z,r)\bigr)\le\sum_{\bm a'\in\Lambda_{k+1}'(z)}\mu\bigl(\cld(\bm a')\bigr) \le\frac{\sharp\Lambda_{k+1}'(z)}{\sharp\Lambda_{k+1}'} \le\frac{9\cst_0^{-1}n_{k+1}^{2\epsilon}}{n_{k+1}^{4-2\epsilon}} =9\cst_0^{-1}n_{k+1}^{-4+4\epsilon}. \]
	Finally, we have $\mu\bigl(B(z,r)\bigr) \le9\cst_0^{-1}r^{2-2\epsilon}$ since $n_{k+1}^{-2}\le r$.
	
	(b) By Lemma~\ref{l:Lam}~\eqref{le:qLam} and~\eqref{le:qLam'}, the condition on~$r$ means that, for all $\bm a'\in\Lambda_{k+1}'$ and $\bm a\in\Lambda_k$,
	\[ |q(\bm a')|^{-2}< r<|q(\bm a)|^{-2}. \]
	Given $z\in E$, we want to find all the cylinders that intersect $B(z,r)$ and of diameter close to~$r$. The bounds above suggest considering the set $\Lambda_k^+(z)$ consisting of sequences with the form $\bm a\bm w$ such that $\bm a\in\Lambda_k$, $\bm a\bm w$ is a prefix of a sequence $\bm a'\in\Lambda_{k+1}'$ and
	\[ \cld(\bm a\bm w)\cap B(z,r)\ne\emptyset,\quad |q(\bm a\bm w)|^{-2}\le r<|q(\bm a\bm w^-)|^{-2}. \]

	For $\bm a\bm w\in\Lambda_k^+(z)$, by Lemma~\ref{l:diamea}, $|\cld(\bm a\bm w)|\le2|q(\bm a\bm w)|^{-2}\le2r$, and so $\cld(\bm a\bm w)\subset B(z,3r)$. Using Lemma~\ref{l:diamea} again, we have
	\[ \lm\bigl(\cld(\bm a\bm w)\bigr)\ge\frac{\pi\cst_0}{|q(\bm a\bm w)|^4} >\frac{\pi\cst_0}{(M+1)^4|q(\bm a\bm w^-)|^4} >\frac{\pi\cst_0r^2}{(M+1)^4}. \]
	Here the second inequality follows from Lemma~\eqref{l:props}~\eqref{le:|q-|}. Therefore,
	\begin{equation}\label{eq:Lamk+}
		\sharp\Lambda_k^+(z)\le\frac{\lm\bigl(B(z,3r)\bigr)}{\min_{\bm a\bm w\in\Lambda_k^+(z)}\lm(\cld(\bm a\bm w))} \le\frac{9\pi r^2}{\pi\cst_0(M+1)^{-4}r^2}=9(M+1)^4\cst_0^{-1}.
	\end{equation}
	
	We now turn to bound $\mu\bigl(\cld(\bm a\bm w)\bigr)$ from above. Recall from Lemma~\ref{l:CaQ} that
	\[ \Gamma_M^{\bm w}(Q_{k+1})=\bigl\{\bm b\colon\bm{wb}\in\Gamma_M(Q_{k+1})\bigr\} \le\frac{(M+1)^{24M} \sharp\Gamma_M(Q_{k+1})}{|q(\bm w)|^{4-2/M}}. \]
	Recall that $2/M\le\epsilon$. By Lemma~\ref{l:props}~\eqref{le:|qq|} and Lemma~\ref{l:Lam}~\eqref{le:qLam},
	\[ |q(\bm w)|^{-(4-2/M)}\le\bigl(3|q(\bm a)|\cdot|q(\bm a\bm w)|^{-1}\bigr)^{4-\epsilon} \le(3n_k^{\lambda-1+\epsilon}r^{1/2})^{4-\epsilon} < 81n_k^{4\lambda-4+4\epsilon}r^{2-\epsilon}.  \]
	Let $c=81(M+1)^{24M}$. It follows from~\eqref{eq:mu} that, for all $\bm a\bm w\in\Lambda_k^+(z)$,
	\begin{equation}\label{eq:caw}
		\begin{split}
			\mu\bigl(\cld(\bm a\bm w)\bigr)&=\frac{\sharp\Gamma_M^{\bm w}(Q_{k+1})}{\sharp\Lambda_{k+1}'} \le c\frac{n_k^{4\lambda-4+4\epsilon}r^{2-\epsilon}\cdot\sharp\Gamma_M(Q_{k+1})}{\sharp\Lambda_k\cdot\sharp\Gamma_M(Q_{k+1})} \le c\frac{n_k^{4\lambda-4+4\epsilon}r^{2-\epsilon}}{n_k^{2\lambda}}\\
			&=cr^{2-\epsilon}n_k^{2\lambda-4+4\epsilon} \le cr^{2-\epsilon-\frac{\lambda-2+2\epsilon}{\lambda-1}} \le cr^{\frac{\lambda}{\lambda-1}-3\epsilon}.
		\end{split}
	\end{equation}
	Here we use Lemma~\ref{l:Lam}~\eqref{le:NLam} and the condition $r<n_k^{-2\lambda+2-2\epsilon}<n_k^{-2\lambda+2}$.
	
	Finally, since $\mu\bigl(B(z,r)\bigr)\le\sum_{\bm a\bm w\in\Lambda_k^+(z)}\mu\bigl(\cld(\bm a\bm w)\bigr)$, the conclusion follows from~\eqref{eq:Lamk+} and~\eqref{eq:caw}.
	
	(c) For $z\in E$, let
	\[ \Lambda_k(z)=\{\bm a\in\Lambda_k\colon \cld(\bm a)\cap B(z,r)\ne\emptyset\}. \]
	In order to get the best upper bound for $\mu\bigl(B(z,r)\bigr)$, we need to use two methods to bound $\sharp\Lambda_k(z)$ from above. The first one is the same as that in the proof of~\eqref{le:rnk+1}. By Lemma~\ref{l:Lam}~\eqref{le:qLam}, $|\cld(\bm a)|\le2n_k^{-2\lambda+2}$ for $\bm a\in\Lambda_k$. Since $r\ge n_k^{-2\lambda+2-2\epsilon}$, we have
	\[ \bm a\in\Lambda_k(z)\implies \cld(\bm a)\subset B(z,3n_k^{2\epsilon}r). \]
	By the lower bound for $\lm\bigl(\cld(\bm a')\bigr)$ in Lemma~\ref{l:Lam}~\eqref{le:qLam'},
	\[ \sharp\Lambda_k(z)\le\frac{\lm(B(z,3n_k^{2\epsilon}r))}{\min_{\bm a\in\Lambda_k}\lm(\cld(\bm a))} \le\frac{9\pi n_k^{4\epsilon}r^2}{\pi\cst_0n_k^{-4\lambda+4-4\epsilon}} =9\cst_0^{-1}n_k^{4(\lambda-1+2\epsilon)}r^2. \]
	
	The second method is to apply Lemma~\ref{l:Lam}~\eqref{le:dtcld}. Since $r<n_k^{-2}$, it implies that, if $\bm a_1,\bm a_2\in\Lambda_k(z)$, then $\bm a_1^-=\bm a_2^-$. Hence, by~\eqref{eq:QI(r)} and~\eqref{eq:Lam}, we have 
	\[ \sharp\Lambda_k(z)\le\sharp I(r_k)\le5n_k^{2\lambda-4+2\epsilon}. \]
	Note that $\cst_0<1$ and $r<n_k^{-2}$. From the above two upper bounds,
	\begin{align*}
		\sharp\Lambda_k(z)&\le\min\big\{9\cst_0^{-1}n_k^{4(\lambda-1+2\epsilon)}r^2, 5n_k^{2\lambda-4+2\epsilon} \big\}\le9\cst_0^{-1}n_k^{8\epsilon}\min\big\{n_k^{4\lambda-4}r^2, n_k^{2\lambda-4} \big\}\\
		 &\le9\cst_0^{-1}n_k^{8\epsilon}(n_k^{4\lambda-4}r^2)^{2/\lambda}(n_k^{2\lambda-4})^{1-2/\lambda}=9\cst_0^{-1}r^{4/\lambda-4\epsilon}n_k^{2\lambda}. 
	\end{align*}
	It follows from~\eqref{eq:mu} and Lemma~\ref{l:Lam}~\eqref{le:NLam} that
	\[ \mu\bigl(B(z,r)\bigr)\le\frac{\sharp\Lambda_k(z)}{\sharp\Lambda_k} \le\frac{9\cst_0^{-1}r^{4/\lambda-4\epsilon}n_k^{2\lambda}}{n_k^{2\lambda}} \le 9\cst_0^{-1}r^{4/\lambda-4\epsilon}. \qedhere \]
\end{proof}

\begin{proof}[Proof of Theorems~\ref{t:result} and~\ref{t:Wpsi}]
	It is clear that $2>\lambda/(\lambda-1)>4/\lambda$ for $\lambda>2$. So Lemma~\ref{l:meaB} implies that, for all $z\in E$,
	\[ \liminf_{r\to0}\frac{\log\mu(B(z,r))}{\log r}\ge4/\lambda-4\epsilon, \quad \limsup_{r\to0}\frac{\log\mu(B(z,r))}{\log r}\ge2-2\epsilon. \]
	Hence, the mass distribution principle (\cite[Proposition~2.3]{Falco97}) gives 
	\[ \hdim E\ge4/\lambda-4\epsilon \quad\text{and}\quad \pdim E\ge2-2\epsilon. \]
	Since $E(\psi)\supset E$ (by Lemma~\ref{l:E}) and $\epsilon$ is arbitrary, we get 
	\[ \hdim E(\psi)\ge4/\lambda \quad\text{and}\quad \pdim E(\psi)\ge2. \]
	Combining this and the upper bounds for dimensions of~$W(\psi)$ in Section~\ref{ss:UB}, we complete the proof since $E(\psi)\subset W(\psi)$.
\end{proof}

\end{document}